\documentclass[12pt]{article}

\usepackage{amsfonts}
\usepackage [dvips]{graphics}
\usepackage[centertags]{amsmath}
\usepackage{amsfonts}
\usepackage{amssymb}
\usepackage{amsthm}
\usepackage{newlfont}
\usepackage{mathrsfs}
\usepackage [dvips]{graphics}
\usepackage[centertags]{amsmath}
\usepackage{amsfonts}
\usepackage{amssymb}
\usepackage{amsthm}
\usepackage{newlfont}
\usepackage{hyperref}

\newlength{\defbaselineskip}
\newcommand{\setlinespacing}[1]%
           {\setlength{\baselineskip}{#1 \defbaselineskip}}

\theoremstyle{plain}
\newtheorem{thm}{Theorem}[section]
\newtheorem{cor}[thm]{Corollary}
\newtheorem{lem}[thm]{Lemma}
\newtheorem{pro}[thm]{Problem}

\theoremstyle{definition}

\newtheorem{ass}{Assumption}[section]
\newtheorem{rmk}{Remark}[section]

\newcommand{\la}{\langle}
\newcommand{\ra}{\rangle}

\makeatletter\@addtoreset{equation}{section} \makeatother

\begin{document}
\title{ General Linear Quadratic Optimal Stochastic Control Problem Driven
by a Brownian Motion and a Poisson Random Martingale Measure with Random Coefficients  \thanks{
 This work is partially supported by the National
Basic Research Program of China (973 Program) (Grant
No.2007CB814904), the National Natural Science Foundation of China
(Grants No.10325101, 11071069), the Specialized Research Fund for
the Doctoral Program of Higher Education of China (Grant
No.20090071120002) and the Innovation Team Foundation of the
Department of Education of Zhejiang Province (Grant No.T200924)..}}

\date{}

\author{Qingxin Meng$$\hspace{1cm}
\\ \small{School of Mathematical Sciences, Fudan University,
Shanghai 200433, China}
\\ \small{Email: 071018034@fudan.edu.cn}}

\maketitle

\begin{abstract}
Consider 
the minimization
of the following quadratic 
functional 

\begin{equation*}\label{eq:2.2}
J(u)=\displaystyle E\int_0^T\big[\langle Q_tX_t, X_t\rangle dt +\langle N_tu_t,
u_t\rangle \big]dt+E\langle MX_T,  X_T\rangle ,
\end{equation*}
where $X$ is the strong  solution to the linear 
state equation driven by a multidimensional Browinan
motion $W$ and a Poisson random martingale measure $\tilde \mu (d\theta,dt)$

\begin{equation*}\label{eq:1.1}
\left\{\begin {array}{ll}

  dX_t=&(A_tX_t+B_tu_t)dt+\displaystyle\sum_{i=1}^d(C_t^{i}X_t+D_t^{i}u_t)dW^{i}_t
  \\&+\displaystyle
  \int_{ Z}(E_t(\theta)X_{t-}+F_t(\theta)u_t)\tilde{\mu}(d\theta,
  dt),
   \\x_0~=&x.
   
   \end {array}
\right. \end{equation*}

Here $u$ is a square integrable adapted control process. The 
problem is conventionally 
called the stochastic 
linear quadratic (LQ in short form)
optimal control problem. This paper is 
concerned the 
following general 
case:
the coefficients $A,B,C^i, D^i, E, F, Q, N$ and $M$
are allowed to be predictable processes or random matrices.
Associated with this LQ problem, the corresponding 
Riccati equation is a multidimensional backward 
stochastic differential equation driven by 
the Brownian motion $W$ and the 
Poisson random martingale measure
$\tilde \mu(d\theta,dt)$ (see \eqref{eq:4.9}).
The  backward stochastic Riccati  differential equation with jumps
will be  abbreviated as BSRDEJ. The generator of  
BSRDEJ is highly nonlinear in the the three
unknown variables $K, L$ and $H$ (see \eqref{eq:4.9}).

In the paper, we will establish the
 connections of the multidimensional BSRDEJ to the
 stochastic
 LQ problem and to the associated Hamilton systems.
 By the connections, we show the optimal control have
 the state feedback representation.
 Moreover,
 we will show the existence and uniqueness  result of the  multidimensional
 BSRDEJ
 for the case where the generator is  bounded linear dependence with respect to
  the unknown  martingale term $ L$ and $H$.

\end{abstract}


\maketitle

\section{Introduction}
  Linear Quadratic (LQ in short form) optimal control problem is
 is a problem where the
system dynamics are linear in state and control variables and the cost functional is
quadratic in the two variables.  It is
well known that LQ problem  is one of the most important classes of optimal control problem,
and the solution of this problem has had a profound impact on many engineering
applications and mathematical finance.

   The very first attempt in tracking deterministic
   LQ problem was made by Bellman, Glicksberg and Gross \cite{BGG}in 1958.
   However, Kalman\cite{KALM60} has been wildly credited for his pioneering
   work published in 1960, in solving the problem  in a linear state
   feedback control form.  Since then, the
   problem has been
   extensively studied and developed in major
   research field in control theory.
  Extension to stochastic
LQ control was first carried out by Wonham \cite{WONH68}.    Bismut \cite{Bism762} performed a detailed
analysis for stochastic LQ control with random coefficients. With the joint effort of
many researchers in the last 50 years, there has been an enormously rich theory on
LQ control, deterministic and stochastic alike (see \cite{RAZH},\cite{CLZ98},\cite{CSYJ01},\cite{ChZh00},\cite{KoZh00},\cite{HuZh05},\cite{YaZhZh04}).

  One of the elegant features of the LQ theory is that it is able to give in explicit
forms the optimal state feedback control and the optimal cost value through the
celebrated Riccati equation.  Associated with deterministic LQ problem or stochastic
LQ problem with deterministic coefficients, the corresponding
Riccati equation is backward deterministic ordinary  differential equation.
For the deterministic Riccati equation was essentially solved by Wonhan \cite{WONH68} by
applying Bellman's principle of
quasilinearizatin (see Bellman\cite{Bell55}) and a monotone convergence result of
symmetric matrices.

 But associated with stochastic LQ problem with random coefficients, the
 corresponding Riccati equation
 is a  highly nonlinear backward stochastic differential equations
 where the generator depends on the  unknown variable in quadratic way.
This  sort of Riccati equation  is called backward stochastic Riccati  equation (BSRDE in short form).
 The interest of proving existence and uniqueness results for such a class of equations
was first addressed by Bismut in \cite{Bism762}. It was clear from the beginning that to
study those highly nonlinear backward stochastic differential equation (BSDE in short form) was already a
challenging task and turned  out to become a long-standing problem. The difficulty comes essentially from the fact that, in its general formulation,
the BSRDE involves quadratic terms in both the unknowns (in particular in the
so-called martingale term). Moreover the nonlinearity can be well defined only in a
subset of the space of nonnegative matrices (where the equation naturally exists).

 For the special case that the generator of BSRDE  depends on
the unknowns martingale term only in linear way,   Bismut\cite{Bism762} obtained the
existence and uniqueness result by constructing a contraction mapping
and the using a fixed point theorem and in 1992, Peng\cite{Peng92a}
also gave a nice treatment on the proof of existence and
uniqueness  by using Bellman's linearization and a monotone
convergence result of symmetric matrices-a generalization of Wonham's
approach to the random situation.
 Later Kohlmann and Tang  have made some progress towards
 solving the open problem. See \cite{KoTa03, KoTa2003} and the references therein. However it
 is still far from  the complete solution. Until 2003, by the methods of stochastic flows, Tang \cite{Tang2003}
  solved the
 long standing open problem of
 the proof of the existence and uniqueness
of the solution of the BSRDE in the general case corresponding to a linear quadratic problem with random coefficients and state-and control-dependent noise.
In this work\cite{Tang2003}, Tang  provides a rigorous derivation
between the Riccati equation and the stochastic
Hamilton system as two different
but equivalent tools for the stochastic LQ problem.

  For the discontinuous LQ problem, in 2003,
Wu and Wang \cite{WuWa03} discussed the stochastic LQ problem  with the
system driven by Brownian motion  and Poisson jumps
and obtain the existence and uniqueness result of a class of deterministic
Riccati equation.
And in 2008, Hu and {\O}ksendal \cite{HuOk} studied
the stochastic LQ problem for the one-dimensional
case with Poisson jumps and random coefficients under partial information,
and the main result is to
show the optimal control
has  state feedback representation by an one-dimensional
BSRDE with jumps  in view of the
technique
  of completing squares. But in \cite{HuOk},
the author did not discussed the
existence and uniqueness of the solution to BSRDE with jumps.

So for the LQ problem with jumps, it
 is still far from  the complete solution.
  The main purpose of this paper is to   discuss detailed  the stochastic LQ control problem
  with random coefficients
  where the linear system is a multidimensional stochastic differential
  equation driven by
 a multidimensional Brownian
 motion and a Poisson random martingale measure.
 In the paper, we will establish the
 connections of the multidimensional Backward stochastic
 Riccati equation with jumps (BSRDEJ in short form) to the
 stochastic
 LQ problem and to the associated Hamilton systems.
 By the connections, we show the optimal control have
 the state feedback representation.
 Moreover,
 we will show the existence and uniqueness  result of the  multidimensional
 BSRDEJ
 for the case where the generator is  bounded linear dependence with respect to
  the unknowns  martingale term.

The rest of the paper is organized as follows. In section 2
we introduce useful notation
and some existing results on stochastic differential equations (SDEs in short form) and BSDEs driven
by Poission random martingale measure. In section 3, we state the stochastic LQ problem we study,
give needed assumptions and prove some preliminary property on the functional cost.
Moreover, we have showed the stochastic  LQ problem with jumps has a unique optimal control.
In section 4,  we  establish   the dual characterization of the optimal
control  by stochastic  Hamilton system.
In section 5, we will present the main results. 
In this section, we will introduce BSRDEJ  and establish the link
with the stochastic Hamilton system with jumps, then show the optimal control of the stochastic LQ problem has
 state feedback representation. In the end, we will focus on discussing
 the existence and uniqueness of the solution to BSRDEJ.

\section{Notation and Preliminaries }

~~~~Throughout this paper, we let
$(\Omega,{\cal F}, \{{\cal F}_t\}_{t\geq 0}, P)$ be a complete
filtered probability space. In this probability space, there is  a
d-dimensional standard Brownian motion  $\{{W_t}\}_{t\geq 0}$ and
a stationary Poisson point process $\{\eta_t\}_{t\geq 0}$
defined  on a fixed
nonempty measurable subset ${Z}$ of $R^1$. We denote by$\mu(de,dt)$
 the counting measure induced by $\{\eta_t\}_{t\geq 0}$ and by  $\nu(d\theta)$ the corresponding
 characteristic measure.  Furthermore,  We assume that
$\nu({Z})<\infty$. Then the compensate random martingale measure is denoted by
  $\tilde{\mu}(d\theta, dt):={\mu}(d\theta,
dt)-\nu(d\theta)dt.$   We  can assume that $\{{\cal
F}_t\}_{t\geq 0}$ is the P-augmentation of the natural filtration
generated by $\{{W_t}\}_{t\geq 0}$ and $\{\eta_t\}_{t\geq 0}$. Denote by $\mathscr{P}$  the
predictable sub-$\sigma$ field of $\mathscr B([0, T])\times
\mathscr{F}$, then we introduce the following notation used
throughout this paper.

$\bullet$~~$H$: a Hilbert space with norm $\|\cdot\|_H$.

$\bullet$~~$\langle\alpha,\beta\rangle:$ the inner product in
$\mathbb{R}^n, \forall \alpha,\beta\in\mathbb{R}^n.$

$\bullet$~~$|\alpha|=\sqrt{\langle\alpha,\alpha\rangle}:$ the norm
of $\mathbb{R}^n,\forall \alpha\in\mathbb{R}^n.$

$\bullet$~~$\langle A,B\rangle=tr(AB^T):$ the inner product in
$\mathbb{R}^{n\times m},\forall A,B\in \mathbb{R}^{n\times m}.$

$\bullet$~~$|A|=\sqrt{tr(AA^*)}:$ the norm of $\mathbb{R}^{n\times
m},\forall A\in \mathbb{R}^{n\times m}$.
Here we denote by $A^*$,  the transpose of a matrix A.

$\bullet$~~$S^n:$the set of all $n\times n$ symmetric matrices.

$\bullet$~~ $S^n_+:$ the subset of all non-negative definite matrices of $S^n.$

$\bullet$~~$(S^n)^l:$ $=\underbrace{S^n\times\cdots\times S^n}_l$.

$\bullet$~~$S_{\mathscr{F}}^2(0,T;H):$ the space of all $H$-valued
and ${\mathscr{F}}_t$-adapted  c\`{a}dl\`{a}g processes
$f=\{f(t,\omega),\ (t,\omega)\in[0,T]\times\Omega\}$ satisfying
$$\|f\|_{S_{\mathscr{F}}^2(0,T;H)}\triangleq\sqrt{E\displaystyle\sup_{0 \leq t \leq T}\|f(t)\|_H^2dt}<+\infty.$$

$\bullet$~~$\cal L_{\mathscr{F}}^2(0,T;H):$ the space of all $H$-valued
and ${\mathscr{F}}_t$-adapted processes $f=\{f(t,\omega),\
(t,\omega)\in[0,T]\times\Omega\}$ satisfying
$$\|f\|_{\cal L_{\mathscr{F}}^2(0,T;H)}\triangleq\sqrt{E\displaystyle\int_0^T\|f(t)\|_H^2dt}<\infty.$$

$\bullet$~~${\cal L}^{\nu,2}( Z; H):$ the space of H-valued measurable
  functions $r=\{r(\theta), \theta \in Z\}$ defined on the measure space $(Z, \mathscr B(Z); v)$ satisfying
$$\|r\|_{{\cal L}^{\nu,2}( Z; H)}\triangleq\sqrt{\displaystyle\int_Z\|r(\theta)\|_H^2v(d\theta)}<~\infty.$$

 $\bullet$~~ ${\cal L}_\cal{F}^{\nu,2}{([0,T]\times  Z; H)}:$ the  space of ${\cal L}^{\nu,2}( Z; H)$-valued
and ${\cal F}_t$-predictable processes $r=\{r(t,\omega,\theta),\
(t,\omega)\in[0,T]\times\Omega\times Z\}$ satisfying
$$\|r\|_{{\cal L}_\cal{F}^{\nu,2}([0,T]\times  Z; H)}\triangleq\sqrt{E\displaystyle
\iint_{Z\times (0, T]}\|r(t,\theta)\|_H^2v (d\theta)dt}<~\infty.$$

$\bullet$~~$L^2(\Omega,{\cal {F}},P;H):$ the space of all
$H$-valued random variables $\xi$ on $(\Omega,{\cal {F}},P)$
satisfying
$$\|\xi\|_{L^2(\Omega,{\cal {F}},P;H)}\triangleq E\|\xi\|_H^2<\infty.$$

Now we give two preliminary lemmas about SDE  and BSDE
driven by the $d$-dimensional Brownian motion
 $W_t$ and the Poisson random
 martingale measure $\tilde \mu(d\theta, dt).$
 which will often
been used in this paper.

\begin{lem}\label{lem:b1}
Let $a$ an ${\cal F}_0$-measurable random variable and
\begin{equation*}
\begin{array}{rl}
b:&[0,T]\times\Omega\times
R^n\longrightarrow R^n,\\
\sigma:&[0,T]\times \Omega\times
R^n\longrightarrow  R^{n\times m },\\
\pi:&[0,T]\times \Omega\times Z\times
R^n\longrightarrow R^n\\
\end{array}
\end{equation*}
are given mappings satisfying the following assumptions
\\(i)$b, \sigma$ and  $\pi$ are measurable with respect to ${\mathscr P}\times
\mathscr{B}(R^n)/\mathscr{B}(R^n), {\mathscr P}\times
\mathscr{B}(R^n)/\mathscr{B}(R^{n\times d}), {\mathscr P}\times
\mathscr{B}( Z)\times \mathscr{B}(R^n)/\mathscr{B}(R^n)$
respectively.
\\(ii) $b(\cdot, 0)\in {\cal L}_{\cal F}^2(0, T;R^n)$; $\sigma(\cdot,
0)\in {\cal L}_{\cal F}^2(0, T;R^{n\times d})$; $ \pi(\cdot, \cdot,
0)\in {\cal L}_{\cal F}^{\nu,  2}([0,  T]\times  Z,
R^{n}).$
\\(iii) $b,\sigma$ and $\pi$ are uniformly Lipschitz  continuous w.r.t.
$x$, i.e. there exists a constant  $C>0$ s.t. for all
$(t,x,\bar{x})\in [0, T]\times \mathbb{R}^n\times \mathbb{R}^n$ and
a.s. $\omega\in\Omega$,
\begin{eqnarray}
\begin{split}\label{eq:b5}
&|b(t, x)-b(t, \bar{x})|^2+|\sigma(t, x)-\sigma(t,  \bar{x})|^2
\\&~~~~~+\int_Z|\pi(t, \theta,  x)
-\pi(t, \theta, \bar{x})|^2\nu (d\theta)\leqslant C|x-\bar{x}|^2.
\end{split}
\end{eqnarray}

Then the SDE with jumps

\begin{eqnarray}
\label{eq:f35}
\begin{split}
  X_t=a+\int_0^tb(s, X_s)ds+\int_0^t\sigma(s, X_s)dW_s+\iint_{Z\times (0,T]
}\pi(s, \theta, X_{s-})\tilde{\mu}(d\theta, ds)
\end{split}
\end{eqnarray}
 has a unique
solution $X\in S_{\cal F}^2(0, T;R^n). $  Moreover, the following
a priori estimate  holds
\begin{eqnarray}\label{eq:f36}
\begin{split}
E\sup_{0\leqslant t\leqslant T}|X_t|^2\leqslant K&\bigg [ E\int_0^T
|b(t, 0)|^2dt+E\int_0^T |\sigma(t, 0)|^2dt\\
&+E\iint_{ Z\times (0, T]}|\pi(t, \theta,
0)|^2\nu(d\theta)dt+E|a|^2\bigg],
\end{split}
\end{eqnarray}
\end{lem}
where  $K$  is a positive  constant depending only on Lipschitz  constant $C$ and $T$.

\begin{lem}\label{lem:b2}
Let $\xi$ an ${\cal F}_T$-measurable random variable and
\begin{equation}
\begin{array}{rl}
f:[0, T]\times \Omega\times R^n\times R^{n\times d}\times {\cal
L}^{\nu, 2}( Z; R^n)\longrightarrow R^n
\end{array}
\end{equation}
is  a given mapping satisfying the following assumptions
\\(i) $f$ is measurable with respect to $\mathscr {P}\times
\mathscr{B}(R^n)\times\mathscr{B}(R^{n\times d})\times \mathscr {B}({\cal L}^{\nu, 2}( Z; R^n))/\mathscr{B}(R^n)$
\\(ii) $f(\cdot, 0, 0, 0)\in {\cal L}_{\cal F}^2(0, T;R^n)$; $\xi\in
{\cal L}^2(\Omega,  {\cal F}, P; R^n).$
\\(iii) $f$ is uniformly Lipschitz continuous w.r.t. $(y,q,r)$, i.e.
there exists a constant $C>0$ s.t. for all
$(t,y,q,r,\bar{y},\bar{q}, \bar{r})\in [0, T]\times
\mathbb{R}^n\times \mathbb{R}^{n\times d}\times
{\cal L}^{\nu, 2}( Z; R^n)\times \mathbb{R}^n\times\mathbb{R}^{n\times
d}\times {\cal L}^{\nu, 2}( Z; R^n)$ and a.s. $\omega\in\Omega$,
\begin{eqnarray}\label{eq:f37}
\begin{split}
|f(t, &y, q, r)-f(t, \bar{y}, \bar{q}, \bar{r})|^2
\\&\leqslant
C\bigg[|y-\bar{y}|^2+|q-\bar{q}|^2 +\int_{ Z
}|r(\theta)-\bar{r}(\theta)|^2 \nu(d\theta)\bigg].
\end{split}
\end{eqnarray} Then the BSDE with jumps
\begin{eqnarray}\label{eq:f38}
\begin{split}
  Y_t =\xi+\int_t^Tf(s, Y_s, Q_s, R_s)ds-\int_t^T Q_sdW_s
 -\iint_{ Z\times (t,  T]}R_s(\theta)\tilde{\mu}(d\theta,  ds)\\
\end{split}
\end{eqnarray}
has a unique solution $$(Y,  Q,  R)\in S_{\cal F}^2(0, T;R^n)\times
{\cal L}_{\cal F}^2(0,  T; R^{n\times d })\times {\cal L}_{\cal
F}^{\nu, 2}([0, T]\times Z; R^n).$$ Moreover, we have the
following a priori  estimate

\begin{eqnarray}\label{eq:b11}
\begin{split}
&~~~E\sup_{0\leqslant t\leqslant
T}|Y_t|^2+E\int_0^T|Q_t|^2dt+E\iint_{ Z\times (0,  T]}|R_t(\theta)|^2 \nu (d\theta)dt\\
&\leqslant K \bigg[E\int_0^T |f(t, 0, 0, 0)|^2dt+E|\xi|^2\bigg],
\end{split}
\end{eqnarray}
where  $K$  is a positive  constant depending only on $C$ and $T$.

Particularly, if
\begin{equation}
\label{eq:f310}
M:=\displaystyle\sup_{\omega\in\Omega}\bigg[\int_0^T|f(t,\omega,  0, 0,
0)|^2dt+|\xi(\omega)|^2\bigg]<\infty,
\end{equation} then  for all $t\in [0, T]$ and a.s., we have
\begin{equation}
\label{eq:f311} |Y_t|^2<M\cdot e^{KT},
\end{equation} where $K$ is
 a positive constant depending only on  Lipschitz constant $C.$
 \end{lem}
\begin{proof}
 The proof of the existence and the uniqueness  can be found in \cite{Tali94}.
 In the following we  will only  proof  the estimate \eqref{eq:f311}. As for the a priori 
 estimate
 \eqref{eq:b11},
 it can be obtained similarly by Gronwall's inequality  and Burkholder-Davis-Gundy inequality.
 In fact,  for any given  $0\leq r\leq
t \leq T$,  applying It\^{o}'s formula to $|y_t|^2$ and takeing
conditional expectation with respect to $ {\cal F}_r$, we have
\begin{eqnarray}
  \begin{split}
&E^{{\cal F}_r}|Y_t|^2+E^{{\cal F}_r}\int_t^T|Q_s|^2ds
+E^{{\cal F}_r}\iint_{ Z\times (t,  T]}|R_s(\theta)|^2\nu(d\theta)ds\\
=&E^{{\cal F}_r}\int_t^T2\langle f(s, Y_s, Q_s, R_s),  Y_s\rangle ds+E^{{\cal F}_r}|\xi|^2\\
\leqslant &E^{{\cal F}_r}\int_t^T2|f(s, Y_s, Q_s, R_s)||Y_s|ds+E^{{\cal F}_r}|\xi|^2\\
\leqslant &E^{{\cal F}_r}\int_t^T2|f(s, Y_s, Q_s, R_s)-f(s,
0, 0, 0)+f(s, 0, 0, 0)||Y_s|ds
+E^{{\cal F}_r}|\xi|^2\\
\leqslant &\frac{1}{2C}E^{{\cal F}_r}\int_t^T|f(s, Y_s, Q_s,
R_s)-f(s, 0, 0, 0)|^2ds+2CE^{{\cal F}_r}\int_t^T|Y_s|^2ds
\\&~~~+E^{{\cal F}_r}\int_t^T|f(s, 0, 0, 0)|^2ds+E^{{\cal F}_r}\int_t^T|Y_s|^2ds
+E^{{\cal F}_r}|\xi|^2\\
\leqslant&(2C+\frac{3}{2})E^{{\cal
F}_r}\int_t^T|Y_s|^2ds+\frac{1}{2}E^{{\cal F}_r}\int_t^T|Q_s|^2ds
+\frac{1}{2}E^{{\cal F}_r}\iint_{ Z\times (t,
T]}|R_s(\theta)|^2\nu(d\theta)ds
\\&+E^{{\cal F}_r}\bigg[\int_t^T|f(s, 0, 0, 0)|^2ds +{{\cal F}_r}|\xi|^2\bigg],
  \end{split}
  \end{eqnarray}
 where  the Lipschitz condition \eqref{eq:f37}
  and the basic inequality  $2ab\leqslant
\beta a^2+\displaystyle\frac{1}{\beta}b^2, ~\forall \beta>0,  a> 0, b> 0
$ are used.
 Therefore, we have

\begin{eqnarray}
  \begin{split}
E^{{\cal F}_r}|Y_t|^2\leqslant &E^{{\cal F}_r}\bigg[\int_0^T|f(s, 0,
0, 0)|^2ds+|\xi|^2\bigg]+(2C+\frac{3}{2})E^{{\cal
F}_r}\int_t^T|y_s|^2ds
\\\leq &M+K\int_t^TE^{{\cal F}_r}|Y_s|^2ds,
  \end{split}
\end{eqnarray}
where we set $L=2C+\displaystyle\frac{3}{2}.$

Consequently, applying Gronwall's inequality, we get

\begin{equation}
       E^{^{{\cal F}_r}}|Y_t|^2\leqslant M e^{K(T-t)}, ~~~~~~ 0\leqslant r\leqslant t
\leqslant T.
     \end{equation}

In the end, particularly taking $r=t$, we obtain the estimate
\eqref{eq:f311}

\end{proof}

\section{Formulation of the problem  and Elementary Results}

Consider the following linear stochastic system derived by
 Brownian motion $W_t$ and Poisson random measure $\tilde{\mu}(d\theta,dt)$

\begin{equation}\label{eq:1.1}
\left\{\begin {array}{ll}

  dX_t=&(A_tX_t+B_tu_t)dt+\displaystyle\sum_{i=1}^d(C_t^{i}X_t+D_t^{i}u_t)dW^{i}_t
  \\&+\displaystyle
  \int_{ Z}(E_t(\theta)X_{t-}+F_t(\theta)u_t)\tilde{\mu}(d\theta,
  dt),
   \\x_0~=&x.

\end {array}
\right. \end{equation}

The process $u$ in  \eqref{eq:1.1}  is our control process.
An admissible control $u$ is defined as a $\{{\cal
F}_t, 0\leq t\leq T\}$-predictable process with values in $R^m$ such that
$E\displaystyle\int_0^T|u(t)|^2dt<+\infty$.  The set of all admissible
control $u$ is denoted by ${\cal A}.$  Note that ${\cal A}$ is a Hilbert
space.

And for any admissible control $u\in {\cal A},$ we consider the following
quadratic cost functional
\begin{equation}\label{eq:2.2}
J(u)=\displaystyle E\int_0^T\big[\langle Q_tX_t, X_t\rangle dt +\langle N_tu_t,
u_t\rangle \big]dt+E\langle MX_T,  X_T\rangle ,
\end{equation}
where $X$ is the strong  solution to the state equation \eqref{eq:1.1}.

Throughout this paper, we make the following assumptions on the coefficients
$A,B,C^i,D^i,\\E,F,Q,N$ and $M.$

\begin{ass}\label{ass:1.1}
 The matrix processes $A:[0, T]\times\Omega\rightarrow R^{n\times n},
B:[0, T]\times\Omega\rightarrow R^{n\times m}; C^i:[0, T]\times\Omega\rightarrow R^{n\times
n}, D^i:[0, T]\times\Omega\rightarrow R^{n\times m}, i=1, 2,
\cdots, d; E:[0, T]\times\Omega\rightarrow \cal L^{v,2}(Z; R^{n\times n}), F:[0,
T]\times\Omega\rightarrow  \cal L ^{v,2}(Z; R^{n\times m});
Q:[0, T]\times\Omega\rightarrow R^{n\times
n}, N:[0, T]\times\Omega\rightarrow R^{m\times m};$ and
the random matrix
$M:\Omega\rightarrow
R^{n\times n}$
are uniformly bounded and $\{{\cal F}_t, 0\leq t\leq T\}$-predictable or ${\cal F}_T$-measurable.
\end{ass}
\begin{ass}\label{ass:1.2}

The state weighting matrix process Q and the control weighting
matrix process N are
a.s. a.e.  symmetric and nonnegative. The terminal state weighting
random matrix M  is a.s. symmetric and
nonnegative. The control weighting matrix process N is a.s. a.e. uniformly positive,
i.e. $ N(t)\geq \delta I $ for some
positive constant $\delta$ and almost all $(t, \omega) \in [0, T]\times \Omega$.
\end{ass}

  Under Assumption \ref{ass:1.1}, from Lemma \ref{lem:b1}, the system \eqref{eq:1.1}
 admits a  unique solution strong solution, which will be denoted by
  $X^{(x,u)}$ or $X$
  if its dependence on
  admissible control $u$ is clear from  the context.
   Then we call $X$ the
state process corresponding to the control process $u$ and
 $(u; X)$ the
admissible pair. Furthermore, from Assumption \ref{ass:1.2} and the
a priori estimate \eqref{eq:f36},
it is easy to check that
$$ |J(u)|<\infty.$$

Then we can pose the  so-called linear quadratic (LQ) problem.

\begin{pro} \label{pro:2.1}
  Find an admissible control $\bar{u}$  such that
 \begin{equation} \label{eq:b7}
J(\bar{u})=\displaystyle\inf_{u\in {\cal
A}}J(u)
\end{equation}
\end{pro}
 Any  $\bar{u}\in {\cal A}$ satisfying the above is called
  an optimal control process of Problem \ref{pro:2.1}
 and the corresponding state process $\bar X$ is
 called the corresponding optimal state process. We also refer
  to $(\bar{u}; \bar X)$  as an  optimal pair of
 Problem \ref{pro:2.1}.

\begin{lem}\label{lem:b4}
Under Assumptions \ref{ass:1.1}-\ref{ass:1.2}, the cost functional $J$ is strictly convex over $\cal A$. Moreover,
 $\displaystyle\lim_ {||u||_{\cal A}{\rightarrow \infty}}J(u)=+\infty$
\end{lem}

 \begin{proof}
 Under  Assumption \ref{ass:1.2}, by the definition of cost functional
 $J$ (see \eqref{eq:2.2}),  it is easy to check that
 $J$ is a convex functional. Since  the weighting
matrix process $N$ is uniformly strictly positive, we can conclude that $J$
is strictly convex over $\cal A$. Moreover, in view of the nonnegative
 property  of $Q,M$
and the uniformly strictly positive property of $N$, we have
        $$0\geq J(u)\geq \delta E\displaystyle\int_0^T
        |u_t|^2dt=\delta||u||^2_{\cal A}.$$
Therefore, $\displaystyle\lim_ {||u||_{\cal A}{\rightarrow
\infty}}J(u)=+\infty.$
 \end{proof}

\begin{lem}\label{lem:b5}
Under Assumptions \ref{ass:1.1}-\ref{ass:1.2}, the cost functional $J$ is Fr\`{e}chet differentiable over $\cal A$. Moreover, the
corresponding
Fr\`{e}chet derivative $J'$ at any admissible control $u\in \cal A$  is given by
\begin{equation}\label{eq:2.4}
 \langle J'(u),  v \rangle=2E\int_0^T\bigg[\langle Q_tX^{(x,u)}_t,  X^{(0,v)}_t\rangle
+\langle N_tu_t, v_t\rangle \bigg]dt+2E\langle MX^{(x,u)}_T,  X^{(0,v)}_T\rangle , ~~~~\forall v\in{\cal A},
\end{equation}
where
$X^{(0,v)}$ is the solution of the SDE
\eqref{eq:1.1} corresponding to the admissible control $v$ and the initial value
$X_0=0,$ and $X^{(x,u)}$ is the state process corresponding to the control process $u.$
\end{lem}

\begin{proof} For $\forall u , v \in{\cal A}$, we define  $$
 \Delta
J:= J(u +v )-J(u )-2E\int_0^T\big[\langle Q_tX^{(x,u)}_t,
X^{(0,v)}_t\rangle +\langle N_tu_t,  v_t\rangle \big]dt-2E\langle MX^{(x,u)}_T,  X^{(0,v)}_T\rangle ,
$$ Then from the definition  of
cost functional  $J$ (see\eqref{eq:2.2}), we have
\begin{eqnarray}
\begin{split}
\Delta J=&E\displaystyle\int_0^T \bigg[\langle Q_t(X^{(x,u)}_t+X^{(0,v)}_t),
X^{(x,u)}_t+X_t^{(0,v)}\rangle +\langle N_t(u_t+v_t),  u_t+v_t\rangle \bigg]dt \\
&+E\langle M(X^{(x,u)}_T+X^{(0,v)}_T),  X^{(x,u)}_T+X^{(0,v)}_T\rangle  -E\displaystyle\int_0^T
\bigg[\langle Q_tX^{(x,u)}_t,  X^{(x,u)}_t\rangle \\&+\langle N_tu_t,  u_t\rangle \bigg]dt-E\langle MX^{(x,u)}_T,  X^{(x,u)}_T\rangle -2E\int_0^T\bigg[\langle N_tX^{(x,u)}_t,  X^{(0,v)}_t\rangle
\\&+\langle Q_tu_t,  v_t\rangle \bigg]dt-2E\langle MX^{(x,u)}_T,  X^{(0,v)}_T\rangle \\
=&2E\displaystyle\int_0^T \bigg[\langle Q_tX^{(0,v)}_t,  X^{(0,v)}_t\rangle +\langle N_tv_t,
v_t\rangle \bigg]dt +2E\langle MX^{(0,v)}_T,  X^{(0,v)}_T\rangle .
\end{split}
\end{eqnarray}
Then it follows from Assumptions \ref{ass:1.1} and
the a priori estimate
\eqref{eq:f36} that
$$|\Delta J| \leq KE\int_0^T|v_t|^2dt=K||v ||^2_{\cal A}  $$

Consequently, we deduce that
$$\displaystyle\lim_ {|v |_{\cal A}
{\rightarrow 0}}\frac{|\Delta J|}{||v ||_{\cal A}}=0,$$
which implies that $J$ is Fr\'{e}chet differentiable and its
Fr\'{e}chet derivative $J'$ is given by  \eqref{eq:2.4}.
\end{proof}

\begin{thm}\label{them:b1}
Under Assumptions \ref{ass:1.1}-\ref{ass:1.2}, Problem\ref{pro:2.1}
has a unique optimal control  $u.$
\end{thm}

\begin{proof}

In view of the fact that the cost functional $J$ is Fr\`{e}chet
differentiable,  strictly convex and $\displaystyle\lim_
{||u ||_{\cal A}{\rightarrow \infty}}J(u )=+\infty$, the
existence and uniqueness of the optimal control can be directly
obtained by Proposition  2.1.2  in \cite{EkTe76}.
\end{proof}

\begin{thm}\label{thm:2.5}
 Under Assumptions \ref{ass:1.1}-\ref{ass:1.2},  a necessary and sufficient conditions for an admissible control
 $u\in \cal A$
 to be an optimal control  of Problem \ref{pro:2.1} is
  for any admissible control $v \in {\cal A},$
  \begin{equation}\label{eq:b16}
    \langle  J'(u ), v -u \rangle = 0.
  \end{equation}
\end{thm}

\begin{proof}
Since the cost functional $J$ is Fr\`{e}chet
differentiable and strictly convex, according to
Proposition 2.2.1 in \cite{EkTe76}, we conclude  that  a necessary and
sufficient conditions for an admissible control
 $u\in \cal A$
 to be an optimal control of Problem\ref{pro:2.1} is
  for any admissible control $v \in {\cal A},$
  \begin{equation}\label{eq:b17}
    \la J'(u ), v -u \ra\geq 0.
    \end{equation}
Since the above inequality is hold for any $v \in {\cal A}$,
we can replace $v $ in the above inequality by
$2u -v $ and get
\begin{equation}\label{eq:b18}
    \la J'(u ), v -u \ra\leq0.
    \end{equation}
    Thanks to \eqref{eq:b17} and \eqref{eq:b18}, we obtain \eqref{eq:b16}
\end{proof}

\begin{cor}\label{cor:b1}
 Under Assumptions \ref{ass:1.1}-\ref{ass:1.2}, a necessary and sufficient
conditions for an admissible control
 $u\in\cal A$
 to be an optimal control of Problem \ref{pro:2.1} is the $Fr\acute{e}chet$ derivative of
 $J$ at the admissible control $u\in\cal A$ given by

  \begin{equation}\label{eq:b19}
    J'(u )= 0.
  \end{equation}
\end{cor}

\begin{proof}
 In the equality \eqref{eq:b16},  replacing  $v $  by $v +u $,
    we have $ \langle  J'(u ), v  \rangle = 0, \forall v\in \cal A,$
    i.e. $J'(u)=0.$ Thus the equality  \eqref{eq:b16} and
    the equality  \eqref{eq:b19} is equivalent.  So the proof can be completed directly by
    Theorem \ref{thm:2.5}.

\end{proof}

\section{ Stochastic Hamilton Systems}

This section will focus on establishing    the dual characterization of the optimal
control  by  stochastic  Hamilton system.

Let $(u , X )$ be an admissible pair,  then the  corresponding
adjoint BSDE of the stochastic
systems \eqref{eq:1.1} is defined by

\begin {equation}\label{eq:f322}
\left\{\begin{array}{lll}
dp_t&=&-\bigg[A_t^*p_t+\displaystyle\sum_{i=1}^dC_t^{i*}q_t^{i}+\displaystyle\int_ Z
E_t^*(\theta)r_t(\theta)\nu (d\theta)+2Q_tX_t\bigg]dt
\\&&+\displaystyle\sum_{i=1}^dq_t^idW_t^{i}+\displaystyle\int_ Z r_t(\theta)\tilde{\mu}(d\theta, dt),
 \\
p_T&=&2MX_T,
\end{array}
  \right.
  \end {equation}

Note that under Assumption \ref{ass:1.1}, from  Lemma \ref{lem:b2}, we  see that the  equation \eqref{eq:f322} admits a  unique solution $$
(p, q, r)\in S_{\cal F}^2(0, T;R^n)\times {\cal L}_{\cal F}^2(0,
T;R^n)\times {\cal L}_{\cal F}^{\nu, 2}([0,
T]\times Z;R^n).
$$

 We define  the
Hamiltonian function $H:[0, T]\times R^n\times R^m\times R^n\times
R^{n\times d}\times {\cal L}^{\nu,  2}( Z; R^n)\longrightarrow
R$ by
\begin{eqnarray}
\begin{split}
&H(t, x, u, p, q, r )\\
=&\big\langle p, A_tx+B_tu \big\rangle +\sum_{i=1}^d\big\langle q^{i},
C_t^{i}x+D_t^{i}u\big\rangle +\int_Z\big\langle r(\theta),
E_t(\theta)x+F_t(\theta)u\rangle  \nu (d\theta)\\
&+\langle Q_tx, x\rangle +\langle N_tu, u\rangle .
\end{split}
\end{eqnarray}
Then we can rewrite the adjoint equation \eqref{eq:f322} in Hamiltonian
system's form:
   \begin {equation}\label{eq:b21}
  \left\{\begin{array}{lll}
dp_t&=&-H_x(t, X_t, u_t, p_t, q_t, r_t )dt+\displaystyle\sum_{i=1}^dq_t^idW_t^i+\int_ Zr_t(\theta)\tilde{\mu}(d\theta, dt), \\
p_T&=&2MX_T.
  \end{array}
  \right.
  \end {equation}

  Now we give the  the dual characterization of the optimal
control.

\begin{thm}\label{thm:b2}Let Assumptions \ref{ass:1.1}-\ref{ass:1.2} be satisfied. Then,
 a necessary and
sufficient condition for an admissible pair $(u, X)$ to be an
optimal pair of Problem \ref{pro:2.1}  is
\begin{equation}\label{eq:24}
H_u(t, X_{t-}, u_t, p_{t-}, q_t, r_t )=0, ~~a.e. a.s.,
\end{equation}
i.e.,
\begin{equation}
\label{eq:b23}2 N_tu_t+B_t^*p_{t-}+\sum_{i=1}^d
D_t^{i*}q^{i}_t+\int_ Z F_t^*(\theta)r_t(\theta)\nu
(d\theta)=0,  ~~a.e. a.s..
\end{equation}
Here $(p,q, r)$  is the  solution
of the adjoint equation \eqref{eq:f322} corresponding to the
admissible pair $(u, X)$.
\end{thm}
\begin{proof}
By  Corollary \ref{cor:b1}, in order to prove Theorem \ref{thm:b2}, we only need to show the equality
\eqref{eq:b19} and  the equality \eqref{eq:24} or \eqref{eq:b23} are equivalent.
Indeed, let $(u, X)$ is an admissible pair.
   From lemma \ref{lem:b5}, for any admissible
control $v\in {\cal A},$ we have
\begin{eqnarray}\label{eq:3.6}
\begin{split}
\langle  J'(u ),  v \rangle =2E\int_0^T\bigg[\langle N_tX_t,  X^{(0,v)}_t)
+\langle Q_tu_t, v_t\rangle \bigg]dt+2E\langle MX_T,  X^{(0,v)}_T\rangle.
\end{split}
\end{eqnarray}

On the other  hand, recalling the  adjoint equation \eqref{eq:f322}
and the state equation \eqref{eq:1.1}, applying It\^{o}'s formula to $\langle X^{(0,v)}_t,p_t\rangle$  and taking expectation, we have
\begin{eqnarray}\label{eq:f329}
\begin{split}
&2E\langle MX_T, X^{(0,v)}_T\rangle =E\langle p_T, X^{(0,v)}_T\rangle \\
=&E\int_0^T\langle p_s,  A_sX_{s}^{(0,v)}+B_sv_s\rangle ds+\displaystyle
\sum_{i=1}^dE\int_0^T\langle q_s^i,  C_s^iX^{(0,v)}_s+D_s^iv_s\rangle ds\\
&+E\int_0^T\langle X_{s}^{(0,v)},
-A_s^*p_s-\sum_{i=1}^dC_s^{i*}q_s^i-\int_ Z
E_s^*(\theta)r(\theta)\nu(d\theta)-2Q_sX_s\rangle ds\\
&+E\iint_{ Z\times(0, T]}\langle
r_s(\theta),  E_s(\theta)X^{(0,v)}_s+F_s(\theta)v_s\rangle \nu (d\theta)ds\\
=&E\int_0^T\langle B_s^*p_s+\sum_{i=1}^dD_s^{i*}q_s^i+\int_{ Z}
F_s^*(\theta)r_s(\theta)\nu(d\theta), v_s\rangle ds\\
&-\displaystyle 2E\int_0^T\langle Q_sX_s, X^{(0,v)}_s\rangle ds,
\end{split}
\end{eqnarray}

Hence

\begin{eqnarray}\label{eq:b26}
\begin{split}
&2E\langle MX_T, X_T^{(0,v)}\rangle +2\displaystyle E\int_0^T\langle Q_tX_s,
X_s^{(0,v)}\rangle ds+\displaystyle
2E\int_0^T\langle N_su_s, v_s\rangle ds\\
=&\displaystyle
E\int_0^T\langle B_s^*p_s+\sum_{i=1}^dD_s^{i*}q_s^i+\int_ Z
F_s^*(\theta)r_s(\theta)\nu(d\theta)+2N_su_s, v_s\rangle ds.
\end{split}
\end{eqnarray}

Combining \eqref{eq:3.6} and\eqref{eq:b26}, we get
\begin{eqnarray}\label{eq:29}
\begin{split}
&\langle J'(u),  v \rangle \\
=&E\int_0^T\langle
B_s^*p_s+\sum_{i=1}^dD_s^{i*}q_s^i+\int_ Z
F_s^*\langle \theta)r_s(\theta)\nu(d\theta)+2N_su_s, v_s\rangle ds\\
=&E\displaystyle\int_0^T\langle H_u(s, X_{s-}, u_s, p_{s-}, q_s, r_s),
v_s\rangle ds, ~~~\forall v\in {\cal A}.
\end{split}
\end{eqnarray}

Since the $v\in {\cal A} $ in \eqref{eq:29} is arbitrary, we deduce that the equality
\eqref{eq:b19} and  the equality \eqref{eq:b23} or \eqref{eq:24} are equivalent. Then
the desired result then follows.
\end{proof}

\begin{cor}\label{cor:22}
 Let assumptions \ref{ass:1.1}-\ref{ass:1.2} be satisfied. Then,  Problem\ref{pro:2.1} has a unique optimal control pair $(u,X),$
  where the optimal control $u$ have
  the dual representation
  \begin{equation}\label{eq:b311}
u_t=-\frac{1}{2}N_t^{-1}\bigg[B_t^*p_{t-}+\sum_{i=1}^dD_t^{i*}q^i_t+\displaystyle\int_ Z
F_t^*(\theta)r_t(\theta)\nu(d\theta)\bigg],
0\leq t\leq T.
 \end{equation}
 Here $(p,q, r)$  is the  unique   solution
of the adjoint equation \eqref{eq:f322} corresponding to the optimal
control  pair $(u, X).$
 \end{cor}

 \begin{proof}
   From Theorem \ref{them:b1}, we know that Problem \ref{pro:2.1} have an
 unique optimal control pair$(u,X).$ Furthermore,  by Theorem \ref{thm:b2}
 and the equality \eqref{eq:b23},  the optimal control is given by
 \eqref{eq:b311}.
 \end{proof}

Now we can introduce the following  so-called stochastic Hamilton system which consists of
the state equation \eqref{eq:1.1},  the dual equation \eqref{eq:f322} and the dual representation
\eqref{eq:b311} by

\begin {equation}\label{eq:f333}
\left\{\begin{array}{lll}
dX_t&=&(A_tX_t+B_tu_t)dt+\displaystyle\sum_{i=1}^d(C_t^{i}X_t+D_t^{i}u_t)dW^{i}_t
  \\&&+\displaystyle
  \int_{ Z}(E_t(\theta)X_{t-}+F_t(\theta)u_t)\tilde{\mu}(d\theta,
  dt), \\
u_t&=&-\displaystyle\frac{1}{2}N_t^{-1}\bigg[B_t^*p_{t-}+\displaystyle\sum_{i=1}^dD_t^{i*}q^i_t+\displaystyle\int_ Z
F_t^*(\theta)r_t(\theta)\nu(d\theta)\bigg], \\
dp_t&=&-\bigg[A_t^*p_t+\displaystyle\sum_{i=1}^dC_t^{i*}q_t^{i}+\displaystyle\int_Z
E_t^*(\theta)r_t(\theta)\nu (d\theta)+2Q_tX_t\bigg]dt,
\\&&+\displaystyle\sum_{i=1}^dq_t^idW_t^{i}+\displaystyle
\int_Z r_t(\theta)\tilde{\mu}(d\theta, dt), \\
X_0&=&x,  ~~p_T=2MX_T.
\end{array}
  \right.
  \end {equation}

Clearly it is a  fully coupled  forward-backward stochastic differential equations
 (FBSDEs in short form) driven by Brownian motion $W$ and Poisson random martingale measure $\tilde{\mu}(d\theta, dt)$. The solutions  consist of the stochastic
 process quaternary  $(X, p, q, r).$

\begin{thm}
  Let assumptions \ref{ass:1.1}-\ref{ass:1.2}
   be satisfied. Then the stochastic Hamilton
  system \eqref{eq:f333} has a unique solution $(X, p, q, r)\in S_{\cal F}^2(0, T;R^n)\times S_{\cal F}^2(0, T;R^n)\times {\cal L}_{\cal F}^2(0, T;R^{n\times d})
 \times {\cal L}_{\cal F}^{\nu , 2}([0, T]\times Z; R^n).$  And $u$ in \eqref{eq:f333} is the
    optimal control of the stochastic LQ Problem \ref{pro:2.1}, the stochastic process
    $X$ is the corresponding optimal state. Moreover,  the following a priori
  estimate holds

\begin{eqnarray}
\label{eq:f334}
  \begin{split}
&\displaystyle E\sup_{0\leqslant t\leqslant T}|X_t|^2+E\sup_{0\leqslant t\leqslant
T}|p_t|^2 +E\int_0^T|q_t|^2dt+E\iint_{ Z\times (0,
T]}|r_t(\theta)|^2\nu(d\theta)dt\leqslant K|x|^2,
  \end{split}
\end{eqnarray}
where $K$ is some deterministic positive constant.
\end{thm}

\begin{proof}
 The existence result can be directly obtained by
Corollary \eqref{cor:22}.
 The uniqueness result is obvious once the a priori
  estimate \eqref{eq:f334} holds. Therefore, it remains to prove that
 the a priori
  estimate  \eqref{eq:f334} hold.

  Let $(X, p, q, r)$ is a solution of the
  stochastic Hamilton systems \eqref{eq:f333}. Using  It\^{o}$^,$s formula
  $\langle p_t, X_t\rangle ,$  we get
   \begin{equation}
  \label{eq:b29}
2E\langle MX_T, X_T\rangle +2E\int_0^T\langle N_tu_t, u_t\rangle +2E\int_0^T\langle Q_tX_t,
X_t\rangle dt=E\langle p_0, x\rangle.
  \end{equation}

  In the following, $K$ will denote a generic positive constant
  and might change from line to line.

 For the backward part of the
  stochastic Hamilton systems \eqref{eq:f333},  using the a priori estimate \eqref{eq:b11} for BSDEs, we have
\begin{eqnarray}\label{eq:b30}
\begin{split}
 E\sup&_{0\leqslant t\leqslant T}|p_t|^2+E\int_0^T|q_t|^2dt
+E\iint_{ Z\times (0, T]}|r_t(\theta)|^2\nu(d\theta)dt\\
&\leqslant K\bigg[E\int_0^T|Q_tX_t|^2dt+E|MX_T|^2\bigg]\\
&\leqslant K\bigg[E\int_0^T\langle Q_tX_t, X_t\rangle dt+E\langle MX_T, X_T\rangle \bigg]
\\&\leqslant K E\langle p_0, x\rangle
\\&\leqslant K E | p_0|| x|
\\
&\leqslant \frac{1}{2} E|p_0|^2+K|x|^2
\\& \leqslant\frac{1}{2} E\sup_{0\leqslant t\leqslant T}|p_t|^2+K|x|^2,
\end{split}
\end{eqnarray}
where we have used  the nonnegative property of $Q$ and $M$, the equality \eqref{eq:b29} and the elementary
inequality
$$2ab\leqslant \varepsilon a^2+\frac{1}\varepsilon b^2,  ~~~\forall \varepsilon >0,  a>0,  b>0.$$

 Hence we get
\begin{eqnarray}\label{eq:b31}
\begin{split}
E\sup_{0\leqslant t\leqslant T}|p_t|^2+E\int_0^T|q_t|^2dt+E\iint_{(0,
T]\times Z}|r_t(\theta)|^2vd(\theta)dt \leqslant K|x|^2.
\end{split}
\end{eqnarray}

On the other hand, for the forward part of the
  stochastic Hamilton systems \eqref{eq:f333},  using the a priori estimate \eqref{eq:f36} for SDEs, we have
\begin{eqnarray}
\begin{split}\label{eq:b32}
  \displaystyle E\sup_{0\leqslant t\leqslant T}|x_t|^2&\leqslant
K\bigg[E\int_0^T|u_t|^2dt+|x|^2\bigg]\\
&\leqslant K\bigg[E\int_0^T\langle N_tu_t,  u_t\rangle dt+|x|^2\bigg]
\\&\leqslant K\bigg[E\langle p_0, x\rangle +|x|^2\bigg]
\\&\leqslant K\bigg[E |p_0| |x| +|x|^2\bigg]
\\
 &\leqslant K\bigg[E|p_0|^2+|x|^2\bigg]
 \\&\leqslant K\bigg[E\sup_{0\leqslant t\leqslant T}|p_t|^2+|x|^2\bigg]
\\&\leqslant K|x|^2,
\end{split}
\end{eqnarray}
 where   we have used  the  nonnegative property of $N$, the  equality \eqref{eq:b29}, the elementary
inequality
$2ab\leqslant \varepsilon a^2+\frac{1}\varepsilon b^2,  ~~~\forall \varepsilon >0,  a>0,  b>0$, and the inequality \eqref{eq:b31}.

Combining the inequality \eqref{eq:b31} and  the inequality \eqref{eq:b32},
the inequality \eqref{eq:f334} is directly obtained.
The proof is complete.
\end{proof}

In summary, the stochastic Hamilton system \eqref{eq:f333}
completely characterizes the optimal control of LQ problem. Therefore, solving LQ problem is equivalent to solving the
stochastic Hamilton system, moreover, the unique optimal control  can be given explicitly by
\eqref{eq:b311}.

\section{ Backward Stochastic  Riccati equation with jumps }

Although  the  stochastic Hamilton system \eqref{eq:f333} is  a complete characterization of the stochastic LQ problem,
 it is a fully coupled forward-backward stochastic differential equation.
The solution to \eqref{eq:f333} would be hard to be solved
 so that this characterization is also not satisfactory.
As the stochastic LQ theory in Brownian motion framework (see \cite{Tang2003}), it is natural to connect
the stochastic LQ problem with stochastic Riccati equation.
In this section, we will introduce stochastic Riccati equation with jumps and establish the link
with the stochastic Hamilton system \eqref{eq:f333}, then show the optimal control of the stochastic LQ problem has
 state feedback representation. In the end, we will focus on discussing
 the existence and uniqueness of the solution to the
 stochastic Riccati equation with jumps.

\subsection{Derivation of stochastic  Riccati equation with jumps}

In the following, by
dynamic programming principle, we will derive  the general form of the stochastic Riccati equation with jumps.

  Now consider the following parameterized stochastic LQ problem  on
  the  initial time $t$ and  the initial state $x$:

   The state equation
\begin{equation}
\left\{\begin {array}{ll}

  dX_s=&(A_sX_s+B_su_s)ds+\displaystyle\sum_{i=1}^d(C_s^iX_s+D_s^iu_s)dW^i_s
  \\&+\displaystyle\int_{E}(E_s(\theta)X_{s-}+F_s(\theta)u_s)\tilde{\mu}(d\theta, ds),
     \\X_t~=&x, ~~~~~0\leq s\leq T.

\end {array}
\right. \end{equation}
 The cost functional
\begin{equation}
J(t, x;u):=\displaystyle E^{{\cal F}_t}\int_t^T\bigg[\langle Q_sX_s,
X_s\rangle +\langle N_su_s, u_s\rangle \bigg]ds +E^{{\cal F}_t}\langle MX_T,  X_T\rangle .
\end{equation}
Define the  value function by

\begin{equation}
\Phi_t(x):=\displaystyle\inf_{u\in {\cal A}}J(t, x;
u).
\end{equation}

Then  the value function $\{\Phi_t(x),  t\in [0,  T],  x\in R^n\}$ is a family
of $\{{\cal F}_t,  0\leq t\leq T\}$-adapted processes  with values in $R.$   In general,
for any $x\in R^n$, $\Phi_t(x)$ is not a bounded variation function
with respect to $t.$  So  we can
only expect that $\{\Phi_t(x), t\in [0, T], x\in R^n \}$ is a family
of semimartingales with the decomposition
\begin{eqnarray}\label{eq:f342}
\Phi_t(x)=\langle Mx,
x\rangle +\displaystyle\int_t^T\Gamma_s(x)ds-\sum_{i=1}^d\int_t^T\Lambda_s^i(x)dW_s^i
-\iint_{(t,T]\times Z}\Psi_s(\theta, x)\tilde{\mu}(d\theta, ds).
\end{eqnarray}
Furthermore, suppose
\begin{eqnarray}
\begin{split}\label{eq:Z1}
 \Phi_t(x)&=\langle K_tx, x\rangle ; \\\Lambda_t^i(x)&=\langle L_t^ix, x\rangle , i=1, 2,
\cdots, d; \\\Psi_s(\theta, x)&=\langle H_t(\theta)x, x\rangle , ~~~~t\in[0, T],
x\in R^n, \theta\in  Z,
\end{split}
\end{eqnarray}  where  $K $ is  a  symmetric matrix-valued $\{{\cal F}_t, 0\leq t\leq T\}$-adapted process,  $L^i
(i=1, 2, \cdots, d) $ and $H$ are symmetric matrix-valued $\{{\cal F}_t, 0\leq t\leq T\}$ -predictable processes.
Firstly, using the dynamic programming principle (see\cite{MeTa10}) 
and It\^{o}-Ventzell  formulation with jumps (see\cite{ChTa10}), we deduce that
$\Gamma_t(x)$ in the semimartingale  decomposition \eqref{eq:f342} have the following
expression
\begin{eqnarray}
\begin{split}\label{eq:f343}
  \displaystyle \Gamma_t(x)&=\inf_{u\in
  {R^m}}\bigg\{\langle D\Phi_t(x), A_tx+B_tu\rangle +\frac{1}{2}\sum_{i=1}^d\langle D^2\Phi_t(x), (C_t^ix+D_t^iu)(C_t^ix+D_t^iu)^*\rangle\\
  &~~~+\sum_{i=1}^d\langle D\Lambda_t^i(x), C_t^ix+D_t^iu\rangle +\langle Q_tx, x\rangle +\langle N_tu, u\rangle \\
&~~~+\int_{ Z}[\Phi_t(x+E_t(\theta)x+F_t(\theta)u)-\Phi_t(x)-\langle D\Phi_t(x),
E_t(\theta)x
+F_t(\theta)u\rangle ]\nu(d\theta)\\
  &~~~+\int_{ Z}[\Psi_t(\theta,  x+E_t(\theta)x+F_t(\theta)u)-\Psi_t(\theta,
  x)]\nu(d\theta)\bigg\},
\end{split}
\end{eqnarray}
where $D\Phi_t(x)$ and $D\Lambda_t(x)$ is the gradient of $\Phi_t(x)$ and $\Lambda_t(x)$ with respect to $x$
respectively,
$D^2\Phi_t(x)$ is the Hessian of $\Phi_t(x)$ with respect to $x$.
 Now substituting  the relationship
\eqref{eq:Z1} into \eqref{eq:f343}, we get
\begin{eqnarray}
\begin{split}\label{eq:f344}
  \displaystyle \Gamma_t(x)=&\inf_{u\in
  {R^m}}\bigg\{\bigg\langle x,  \bigg[K_{t}A_t+A^*_tK_{t}+\sum_{i=1}^dL^{i}_tC^{i}_t+
\sum_{i=1}^dC^{i*}_tL^{i}_t
+\sum_{i=1}^dC^{i*}_tK_{t}C^{i}_t\\
&+\int_{ Z}H_t(\theta)E_t(\theta)\nu(d\theta)
+\int_{ Z}
E^{*}_t(\theta)H_t(\theta)\nu(d\theta)\\
&+\int_{ Z} E^{*}_t(\theta)K_{t}E_t(\theta)\nu (d\theta)
+\int_{ Z}E^{*}_t(\theta)H_t(\theta)E_t(\theta)\nu(d\theta)
+Q_t\bigg]x\bigg\rangle\\
&+2\bigg\langle u,   \bigg[B^*_tK_{t}+\sum_{i=1}^dD^{i*}_tL^{i}_t+\sum_{i=1}^dD^{i*}_tK_{t}C^{i}_t\\
&+\int_ Z
F^{*}_t(\theta)H_t(\theta)\nu(d\theta)
+\int_ ZF^{*}_t(\theta)K_{t}E_t(\theta)\nu(d\theta)\\
&+\int_ ZF^{*}_t(\theta)H_t(\theta)E_t(\theta)\nu
(d\theta)\bigg]x\bigg\rangle
+\bigg\langle u,  \bigg[N_t+\sum_{i=1}^dD^{i*}_tK_{t}D^{i}_t\\
&+\int_ Z
F^{*}_t(\theta)K_{t}F_t(\theta)\nu(d\theta)+\int_ Z
F^{*}_t(\theta)H_t(\theta)F_t(\theta)\nu (d\theta)\bigg]u\bigg\rangle \bigg\}.
\end{split}
\end{eqnarray}
Ii is obvious that  for $\forall (t, x,\omega)\in [0, T]\times R^n\times \Omega$,
$\Gamma_t(x)$ is  the Quadratic functional extreme  with respect to $u\in R^m$.

  Furthermore, if
$N_t+\displaystyle\sum_{i=1}^dD^{i*}_tK_tD^i_t+\int_{ Z}
F^*_t(\theta)K_tF_t(\theta)\nu(d\theta)+\int_{ Z}
F^*_t(\theta)H_t(\theta)F_t(\theta)\nu(d\theta)$ is  strictly positive definite,
then it follows that the infimum   in
\eqref{eq:f344} is obtained at
\begin{eqnarray}
\begin{split}\label{eq:f345}
  \displaystyle u&=-\bigg[N_t+\displaystyle\sum_{i=1}^dD_t^{i*}K_tD_t^i+\int_{ Z}
  F_t^*(\theta)K_tF_t(\theta)\nu(d\theta)\\
  &~~~+\int_{ Z} F_t^*(\theta)H_t(\theta)F_t(\theta)\nu (d\theta)\bigg]^{-1}\bigg(B_t^*K_t+
 \displaystyle\sum_{i=1}^d D_t^{i*}K_tC_t^i+\displaystyle\sum_{i=1}^dD_t^{i*}L_t^i\\
  &~~~+\int_{ Z} F_t^*(\theta)K_t(\theta)E_t(\theta)\nu (d\theta)
  +\int_{ Z}F_t^*(\theta)H_t(\theta)\nu (d\theta)
  \\&~~~+\int_{ Z}F_t^*(\theta)H_t(\theta)E_t(\theta))\tilde \nu(d\theta)\bigg)x
\end{split}
\end{eqnarray}

Combining \eqref{eq:f342}, \eqref{eq:Z1},
\eqref{eq:f344} and \eqref{eq:f345}, we deduce that the matrix-valued  processes $(K, L,
H)$ satisfy the following  Riccati equation
\begin{equation}\label{eq:4.9}
\left\{\begin {array}{ll}

 & dK_t=-G_t-Q_t+\hat{B}_t\hat{N}_t^{-1}\hat{B}_t^*dt+\displaystyle\sum_{i=1}^dL_t^idW_t^i+\int_{ Z}
H_t(\theta)\mu(d\theta, dt), \\
&K_T=M,

\end {array}
\right. \end{equation}
where
\begin{eqnarray}
  \begin{split}\label{eq:f347}
G_t:=&\displaystyle
K_tA_t+A^*_tK_t+\sum_{i=1}^dL^i_tC^i_t+\sum_{i=1}^dC^{i*}_tL^{i}_t
+\sum_{i=1}^dC^{i*}_tK_tC^{i}_t\\
&+\int_{ Z}H_t(\theta)E_t(\theta)\nu(d\theta)
+\int_{ Z}
E^*_t(\theta)H_t(\theta)\nu(d\theta)\\
&+\int_{ Z} E^{*}_t(\theta)K_tE_t(\theta)\nu(d\theta)
+\int_{ Z}
E^{*}_t(\theta)H_t(\theta)E_t(\theta)\nu(d\theta),
\end{split}
\end{eqnarray}
\begin{eqnarray}
  \begin{split}
\hat{B}_t=&K_tB_t+\sum_{i=1}^dL^i_tD^i_t+\sum_{i=1}^dC^{i*}_tK_tD^i_t
\\&+\int_{ Z}H_t(\theta)F_t(\theta)\nu(d\theta)
+\int_{ Z}
E^{*}_t(\theta)K_tF_t(\theta)\nu(d\theta)\\
&+\int_{ Z}E^{*}_t(\theta)H_t(\theta)F_t(\theta)\nu(d\theta),
\end{split}
\end{eqnarray}
\begin{eqnarray}\label{eq:4.12}
  \begin{split}
\hat{N}_t=&N_t+\sum_{i=1}^dD^{i*}_tK_tD^{i*}_t
+\int_ Z F^{*}_t(\theta)K_tF_t(\theta)\nu(d\theta)\\
&+\int_ Z
F^{*}_t(\theta)H_t(\theta)F_t(\theta)\nu(d\theta).
  \end{split}
\end{eqnarray}
It is a high order nonlinear backward stochastic differential equations with the generator $-G_t-Q_t+\hat{B}_t\hat{N}_t^{-1}\hat{B}_t^*$, the unknown elements are the triple matrix process
$ (K, L, H).$  The above backward stochastic Riccati  differential equation with jumps
will be hereafter abbreviated as BSRDEJ.

Now we give the rigorous connection of BSRDEJ \eqref{eq:4.9} to the
stochastic Hamilton system \eqref{eq:f333} and to the stochastic LQ Problem \ref{pro:2.1}.

\begin{thm}\label{thm:f341}
Let Assumptions  \ref{ass:1.1}--\ref{ass:1.2} be satisfied. Let  $(X, p, q, r)$ be the
solution of  the
stochastic
Hamilton system \eqref{eq:f333}
with   $u$  being  the optimal control.
Assume that  $(K, L, H)\in {\cal S}_{\cal F}^2(0, T; {\cal S}^n)\times
{\cal L}_{\cal F}^2(0, T; ({\cal S}^n)^d)
 \times {\cal L}_{\cal F}^{\nu, 2}([0, T]\times Z; {\cal S}^n)$
is  the solution to BSRDEJ \eqref{eq:4.9} and the  matrix-valued
process
$\hat{N}$(noting \eqref{eq:4.12}) is  a.e.a.s.  positive definite. Then,
we have, for  $t\in [0, T]$ and $\theta \in Z$
     \begin{eqnarray}\label{eq:4.13}
     \begin{split}
      p_t&=K_tX_t;  \\
    q^i_t&=(L_t^i+K_{t-}C^i_t)X_{t-}+K_{t-}D^i_tu_t,  ~~~~i=1, 2, \cdots, d;\\
     r_t(\theta)&=\bigg(H_t(\theta)+K_{t-}E_t(\theta)+H_t(\theta)E_t(\theta)\bigg)X_{t-}\\
     &~~~~
    +\bigg(K_{t-}F_t(\theta)
    +H_t(\theta)F_t(\theta)\bigg)u_t.
     \end{split}
   \end{eqnarray}
   \end{thm}

   \begin{proof}

    Use It\^{o} formula to compute $K_tx_t$ and compare it with $p_t.$
 The identification of the integrands of Lebesgue  and It\^{o}'s
 integrals yields the desired relation \eqref{eq:4.13}.
   \end{proof}

  Now we give the state feedback representation of optimal control $u$.

\begin{thm}\label{thm:f342}
Let Assumptions \ref{ass:1.1}--\ref{ass:1.2} hold. Let $(u,
X)$ be the optimal pair of the stochastic LQ Problem \ref{pro:2.1}.  Assume that
 $(K, L, H)\in {\cal S}_{\cal F}^2(0, T; {\cal S}^n)\times {\cal L}
_{\cal F}^2(0, T; ({\cal S}^n)^d)
 \times {\cal L}_{\cal F}^{\nu, 2}([0, T]\times Z; {\cal S}^n)$
is the solution  to BSRDEJ \eqref{eq:4.9} and the  matrix-valued
process
$\hat{N}$(noting \eqref{eq:4.12}) is  a.e.a.s.  positive definite. Then $u$ has the following state feedback
representation
\begin{eqnarray}
\begin{split}\label{eq:4.14}
    \displaystyle u_t&=\bigg[N_t+\displaystyle\sum_{i=1}^dD_t^{i*}K_{t-}D_t^i+\int_{ Z}
  F_t^*(\theta)K_{t-}F_t(\theta)\nu(d\theta)\\
  &~~~+\int_{ Z} F_t^*(\theta)H_t(\theta)F_t(\theta)
\nu (d\theta)\bigg]^{-1}\bigg[B_t^*K_{t-}+
 \displaystyle\sum_{i=1}^d D_t^{i*}K_{t-}C_t^i\\&~~~+
\displaystyle\sum_{i=1}^dD_t^{i*}L_t^i+\int_{ Z} F_t^*(\theta)K_{t-}(\theta)E_t(\theta)
\nu (d\theta)+\int_{ Z}
  F_t^*(\theta)H_t(\theta)\nu (d\theta)
  \\&~~~+\int_{ Z}F_t^*(\theta)H_t
(\theta)E_t(\theta))\nu(d\theta)\bigg]X_{t-},  ~~~a.e.a.s..
\end{split}
\end{eqnarray}
Moreover, the following relation holds
 $$
 \displaystyle \inf_{u\in \mathscr {U}}J(u)=E\langle K_0x, x\rangle .
 $$
   \end{thm}

\begin{proof}
Putting into the relationship \eqref{eq:4.13} into the dual representation  \eqref{eq:b23}, we get
the state feedback representation
\eqref{eq:4.14}. Since $(u,X)$ is the optimal pair, combining the relationship
\eqref{eq:b29} and the first relationship  in  \eqref{eq:4.13}, we
 get

\begin{eqnarray}\label{eq:c11}
     \begin{split}
\displaystyle \inf_{u\in \mathscr
{U}}J(u)&=2E\langle Mx_T, x_T\rangle+2E\int_0^T\langle N_tu_t, u_t\rangle dt+2E\int_0^T\langle Q_tx_t, x_t\rangle dt\\
&=E\langle P_0, x\rangle =E\langle K_0x, x\rangle .
     \end{split}
   \end{eqnarray}

The proof is complete.

\begin{rmk}
  Formula \eqref{eq:4.14} provides a characterization
  of the optimal control in the terms of the solution to
  BSRDEJ \eqref{eq:4.9}. BSRDEJ \eqref{eq:4.9} is not
  a coupled equation, and this characterization is
  preferred to \eqref{eq:f333}.
\end{rmk}

\begin{rmk}
Putting \eqref{eq:4.14} into  the second equality and the third equality of \eqref{eq:4.13},
we have
 \begin{eqnarray*}
   \begin{split}
     q^i_t&=(L_t^i+K_{t-}C^i_t)X_{t-}+K_{t-}D^i_t\bigg[N_t+\displaystyle\sum_{i=1}^dD_t^{i*}K_{t-}D_t^i+\int_{ Z}
  F_t^*(\theta)K_{t-}F_t(\theta)\nu(d\theta)\\
  &~~~+\int_{ Z} F_t^*(\theta)H_t(\theta)F_t(\theta)
\nu (d\theta)\bigg]^{-1}\bigg[B_t^*K_{t-}+
 \displaystyle\sum_{i=1}^d D_t^{i*}K_{t-}C_t^i\\&~~~+
\displaystyle\sum_{i=1}^dD_t^{i*}L_t^i+\int_{ Z} F_t^*(\theta)K_{t-}(\theta)E_t(\theta)
\nu (d\theta)+\int_{ Z}
  F_t^*(\theta)H_t(\theta)\nu (d\theta)
  \\&~~~+\int_{ Z}F_t^*(\theta)H_t
(\theta)E_t(\theta))\nu(d\theta)\bigg]X_{t-},  ~~~~i=1, 2, \cdots, d;\\
     r_t(\theta)&=\bigg(H_t(\theta)+K_{t-}E_t(\theta)+H_t(\theta)E_t(\theta)\bigg)X_{t-}\\
     &~~~~
    +\bigg(K_{t-}F_t(\theta)
    +H_t(\theta)F_t(\theta)\bigg)\bigg[N_t+\displaystyle\sum_{i=1}^dD_t^{i*}K_{t-}D_t^i+\int_{ Z}
  F_t^*(\theta)K_{t-}F_t(\theta)\nu(d\theta)\\
  &~~~+\int_{ Z} F_t^*(\theta)H_t(\theta)F_t(\theta)
\nu (d\theta)\bigg]^{-1}\bigg[B_t^*K_{t-}+
 \displaystyle\sum_{i=1}^d D_t^{i*}K_{t-}C_t^i+
\displaystyle\sum_{i=1}^dD_t^{i*}L_t^i\\&~~~+\int_{ Z} F_t^*(\theta)K_{t-}(\theta)E_t(\theta)
\nu (d\theta)+\int_{ Z}
  F_t^*(\theta)H_t(\theta)\nu (d\theta)
  \\&~~~+\int_{ Z}F_t^*(\theta)H_t
(\theta)E_t(\theta))\nu(d\theta)\bigg]X_{t-},~~t\in [0, T],~~\theta\in Z.
   \end{split}
 \end{eqnarray*}

\end{rmk}

\subsection{ Existence and uniqueness of BSRDE with jump}
From Theorem \ref{thm:f342},
we know that the optimal control $u$ of the stochastic LQ Problem
\ref{pro:2.1}
can be expressed by the solution $(K, L, H)$ to the BSRDEJ \eqref{eq:4.9}.
Therefore,  solving stochastic LQ Problem \ref{pro:2.1} is equivalent
to solving the  BSRDEJ \eqref{eq:4.9}.
But the  BSRDEJ \eqref{eq:4.9} is
a high order nonlinear backward stochastic differential equation
with jumps.  And the general theory of BSDE (see lemma \ref{lem:b2}) can be not applied to use to
 guarantee the  existence and uniqueness of its solution.
 Moreover, different from the BSRDE driven only by
 Brownian motion (see Tang \cite{Tang2003}), the  BSRDEJ \eqref{eq:4.9} has also a notable characteristic:   the nonlinear term
$\hat{N}_t^{-1}=(N_t+D_t^{i*}K_tD_t^i+\displaystyle\int_{ Z}
 F_t^*(\theta)K_tF_t(\theta)\nu(d\theta)+\displaystyle\int_{ Z}
 F_t(\theta)H_t(\theta)F_t(\theta)\nu(d\theta))^{-1}$  contains not only the first unknown
 element $K$, but also the third unknown element $H$. For the  BSRDE
 driven by only Brownian motion, the nonlinear term
 $\hat{N}_t^{-1}$ is degenerated  into
 $(N_t+D_t^{i*}K_tD_t^i)^{-1}$ which only contain the first unknown element
$K_t.$  In \cite{Tang2003}, we can proof the   $K_t$ is
non-negative matrix, so
$(N_t+D_t^{i*}K_tD_t^i)^{-1}$ is well defined.  For the second unknown element
$L$, we can only  show it's square integrability, but we can not show if it is a  non-negative matrix.
 So for the BSRDEJ \eqref{eq:4.9}, how to guarantee
$(N_t+D_t^{i*}K_tD_t^i+\displaystyle\int_{ Z}
 F_t^*(\theta)K_tF_t(\theta)\nu(d\theta)+
\displaystyle\int_{ Z}
 F_t^*(\theta)H_t(\theta)F_t(\theta)\nu(d\theta))^{-1}$ to be well-defined
 is posed to be a challenging problem.

 In this paper, we  show the existence and uniqueness  result only
 for the case where the generator is a bounded linear dependence with respect to
  the second unknown element $ L $ and the third unknown element $ H $.
For the general case,
we will to continue the discussion in future research .

Now we give the further assumptions on the coefficients of  stochastic system
\eqref{eq:1.1}. Assume that the coefficients
\begin{eqnarray*}
  \begin{split}
  C=&(C^1, \cdots,  C^d)=(C^{11}, \cdots, C^{1d_1},  C^{21}\cdots,  C^{2d_2}),\\
  D=&(D^1, \cdots,  D^d)=(D^{11}, \cdots, D^{1d_1},  0\cdots,  0),  \\
  F=&0,
\end{split}
\end{eqnarray*}
where $d_1+d_2=d.$

In this case the stochastic system \eqref{eq:1.1} is reduced to the following form
\begin{equation}\label{eq:f354}
\left\{\begin {array}{ll}

  dX_t=&(A_tX_t+B_tu_t)dt+\displaystyle\sum_{i=1}^{d_1}C_t^{1i}X_tdW^{1i}_t
+\displaystyle\sum_{i=1}^{d_1}(C_t^{2i}X_t+D_t^{2i}u_t)dW^{2i}_t
  \\&+\displaystyle
  \int_{ Z}E_t(\theta)X_{t-}\tilde{\mu}(d\theta,
  dt)
   \\x_0~=&x,

\end {array}
\right. \end{equation}  Denote by $\{{\cal F}_t^*\}_{t\geqslant
0}$    the P-augmentation of the natural $\sigma$-filtration which is generated by Brownian motion
$(W^{11}, \cdots,
W^{1d_1})$ and Poisson random martingale measure $\tilde{\mu}(d\theta,
dt)$.  In the following we give the further assumptions on adaption of the
 coefficients of the stochastic LQ problem. ÏÂ
\begin{ass}\label{ass:f343}
Assume that  $A,  B,  C,  D,  E, Q,  N$ are uniforming  bounded $\{{\cal
F}_t^*, 0\leq t\leq T\}$-predictable matrix-valued processes. And
the random matrix $M$ is bounded ${\cal F}_T^*$-measurable.
\end{ass}

Under Assumption \ref{ass:f343}, again by  dynamic programming principle and It\^{o}-Ventzell
formulation, BSRDEJ \eqref{eq:4.9} is reduced to the
following form
 \begin{eqnarray}\label{eq:4.17}
    \begin{split}
dK_t=&-\bigg[K_{t}A_t+A^*_tK_{t}+\displaystyle\sum_{i=1}^{d_1}C^{1i*}_tL^{1i}_t
+\displaystyle\sum_{i=1}^{d_1}L^{1i}_tC^{1i}_t+\displaystyle\sum_{i=1}^{d_1}C^{1i*}_tK_{t}C^{1i}_t\\
&+\displaystyle\sum_{i=1}^{d_2}C^{2i*}_tK_{t}C^{2i}_t+\int_ ZH_t(\theta)E_t(\theta)\nu(d\theta)
+\int_ Z
E^{*}_t(\theta)H_t(\theta)\nu(d\theta)\\
&+\int_ Z
E^{*}_t(\theta)K_{t}E_t(\theta)\nu(d\theta)+\int_ Z
E^{*}_t(\theta)H_t(\theta)E_t(\theta)\nu(d\theta)\\
&
+Q_t-\hat{B}(t, K_t)\hat {N}^{-1}(t, K_t)\hat{B}^*(t, K_t)\bigg]dt\\
&+\displaystyle\sum_{i=1}^{d_1}L^{i}_tdW^{1i}_t+\int_ ZH_t(\theta)\tilde{\mu}(d\theta,
dt),
    \end{split}
  \end{eqnarray}
where for $\forall K\in {\cal S}^n,$  we define
\begin{eqnarray}\label{eq:4.18}
    \begin{split}
\hat{B}(t, K):=&KB_t+\displaystyle\sum_{i=1}^{d_2}C^{2i}_tKD^{2i*}, \\
\hat{N}(t, K):=&N_t+\displaystyle\sum_{i=1}^{d_2}D^{2i*}KD^{2i}.
\end{split}
  \end{eqnarray}
In the following we state the existence and uniqueness  result of the solution
BSRDEJ \eqref{eq:4.17}.
\begin{thm}\label{thm:f344}
  Let Assumption \ref{ass:1.2} and Assumption \ref{ass:f343} hold.
  Then BSRDEJ \eqref{eq:4.17}  has a unique solution
  $(K, L, H)\in {\cal S}_{\cal F}^2(0, T; \cal{S}^n)\times {\cal L}_{\cal F}^2(0, T;({\cal S}^n)^{d_1})\times
  {\cal L}_{\cal F}^{\nu, 2}([0, T]\times Z; {\cal S}^n).
 $ Moreover, $K$ is uniformly bounded and nonnegative a.s.a.e..
\end{thm}

In order to show the theorem, we need the following two lemmas.

Consider the following linear BSDE
\begin{equation}\label{eq:4.19}
\left\{\begin {array}{ll}
-d\hat{K}_t=&\bigg[\hat{K}_t\hat{A}_t+\hat{A}^*_t\hat{K}_t
+\displaystyle\sum_{i=1}^{d_1}\hat{L}^{1i}_t\hat{C}^{1i}_t
+\displaystyle\sum_{i=1}^{d_1}\hat{C}^{1i*}_t\hat{L}^{1i}_t
+\displaystyle\sum_{i=1}^{d_1}\hat{C}^{1i*}_t\hat{K}_{t}\hat{C}^{1i}_t\\
&+\displaystyle\sum_{i=1}^{d_2}\hat{C}^{2i*}_t\hat{K}_{t}\hat{C}^{2i}_t
+\displaystyle\int_{ Z}\hat{H}_t(\theta)\hat{E}_t(\theta)\nu(d\theta)
+\displaystyle\int_{ Z}
\hat{E}^{*}_t(\theta)\hat{H}_t(\theta)\nu(d\theta)\\
&+\displaystyle\int_{ Z}
\hat{E}^{*}_t(\theta)\hat{K}_{t}\hat{E}_t(\theta)\nu(d\theta)
+\displaystyle\int_{ Z}
\hat{E}^{*}_t(\theta)\hat{H}_t(\theta)\hat{E}_t(\theta)\nu(d\theta)
+\hat{Q}_t\bigg]dt\\
&-\displaystyle\sum_{i=1}^{d_1}{\hat L}^{1i}_tdW^{1i}_t
-\displaystyle\int_ Z{\hat H}_t(\theta)\tilde{\mu}(d\theta, dt), \\
~~~~~\hat{K}_T=&\hat{M}.
\end{array}
\right.
\end{equation}

\begin{lem}\label{lem:f345}
  Let $\hat{A}, \hat{C}^{1i}( i=1, 2, \cdots, d_1), \hat{C}^{2i} (i=1, 2, \cdots, d_2),\hat{E}$ be $R^{n\times n}$-valued, and  $\hat{Q}$ be
$S^n$-valued, uniformly bounded $\{{\cal F}_t^*, 0\leq t\leq T\}$-predictable process. Let $\hat{M}$ be ${\cal
S}^n$-valued bounded ${\cal F}_T^*$-measurable random variable.
Then BSDE \eqref{eq:4.19} has unique solution $(\hat K, \hat L, \hat H)\in {\cal S}_{{\cal F}^*}^2(0,
T; \cal{S}^n)\times {\cal L}_{{\cal F}^*}^2(0, T;({\cal
S}^n)^{d_1})\times
  {\cal L}_{{\cal F}^*}^{\nu, 2}([0, T]\times Z; {\cal S}^n).
 $ Moreover,
\begin{equation}\label{eq:f358}
    \displaystyle \sup_{(t, \omega)\in[0, T]\times \Omega}|\hat{K}_t(\omega)|^2\leqslant \kappa_0<+\infty,
   \end{equation}
where $k_0$ depends  on
\begin{equation}
\displaystyle \sup_{\omega}
\bigg(|\hat{M}(\omega)|^2+\displaystyle\int_0^T|\hat{Q}_t(\omega)|^2dt\bigg).
\end{equation}
If $\hat{Q}$ and $M$ are nonnegative a.s.a.e, then
$\hat{K}$ is also nonnegative a.s.a.e..

\end{lem}
\begin{proof}

  According to Lemma \ref{lem:b2},  the existence and uniqueness as
well as the inequality \eqref{eq:f358} can be obtained directly.
It remains to prove the nonnegativity of
$\hat{K}.$ For any given $(t,x)\in [0, T]\times R^n $,  we introduce the following linear
SDE:
\begin{eqnarray}\label{eq:f359}
\begin{split}
 dy_s=&\hat{A}_sy_sds+\displaystyle\sum_{i=1}^{d_1}\hat{C}^{1i}_sy_sdW^{1i}_s+\displaystyle\sum_{i=1}^{d_2}
\hat{C}^{2i}_sy_sdW^{2i}_s\\
& +\int_ Z\hat{E}_s(\theta)y_{s-}\tilde{\mu}(d\theta, ds),
~~~y_t=x,  ~~t\leqslant s\leqslant T.
\end{split}
\end{eqnarray}
From Lemma \ref{lem:b1}, SDE \eqref{eq:f359} has a unique strong
solution $y.$  Applying It\^{o}'s formula to $\langle \hat{K}_sy_s, y_s\rangle $ we have
  \begin{eqnarray}\label{eq:f360}
  \begin{split}
  d\langle \hat{K}_s&y_s, y_s\rangle =-\langle \hat{Q}_sy_s,  y_s\rangle ds
   +\displaystyle\sum_{i=1}^{d_1}  \langle y_s, (\hat{L}^{1i}_s+\hat{K}_s\hat{C}^{1i}_s
  +\hat{C}^{1i*}_s\hat{K}_s)y_s  \rangle dW^{1i}_s
  \\&+\displaystyle\sum_{i=1}^{d_1}  \langle y_s, (\hat{K}_s\hat{C}^{2i}_s
   +\hat{C}^{2i*}_s\hat{K}_s)y_s  \rangle dW^{2i}_s\\
   &+\int_ Z  \langle y_{s-}, \big(\hat{H}_s(\theta)
   +\hat{K}_{s-}\hat{E}_s(\theta)+\hat{E}^{*}_s(\theta)\hat{K}_{s-}
   +\hat{H}_s(\theta)\hat{E}_s(\theta)\big)y_{s-}  \rangle \tilde{\mu}(d\theta, ds)\\
  &+\int_ Z  \langle  y_{s-}, \hat{E}^{*}_s(\theta)(\hat{H}_s(\theta)
   +\hat{K}_{s-}\hat{E}_s(\theta)+\hat{H}_s(\theta)\hat{E}_s(\theta))y_{s-}  \rangle \tilde{\mu}(d\theta, ds),
    \end{split}
    \end{eqnarray}
Thus, taking conditional expectation, we get
    \begin{equation}\label{eq:f361}
      \langle \hat{K}_tx, x\rangle =E^{{\cal
F}^*_t}\bigg[\int_t^T\langle \hat{Q}_sy_s, y_s\rangle ds+\langle \hat{M}y_T, y_T\rangle \bigg].
    \end{equation}

     Since $\hat{Q}$ and $\hat{M}$ are nonnegative a.s.a.e., from
\eqref{eq:f361}, we conclude that
   $\hat{K}$ is nonnegative a.s.a.e..
\end{proof}
 Define the mapping
 $F:[0, T]\times  ({\cal S}^n)_+\times({\cal
S}^n)^{d_1}\times {\cal L}^{\nu, 2}( Z;{\cal S}^n)\times
R^{m\times n}\rightarrow {\cal S}^n $ by

  \begin{eqnarray}
    \begin{split}
F(t, K,  L^1,  H,
U)=&(A_t-B_tU)^*K+K(A_t-B_tU)+\displaystyle\sum_{i=1}^{d_1}C^{1i*}_tL^{1i}
+\displaystyle\sum_{i=1}^{d_1}L^{1i}C^{1i}_t\\
&+\displaystyle\sum_{i=1}^{d_1}C^{1i*}_tKC^{1i}_t+\int_ ZH(\theta)E_t(\theta)\nu(d\theta)
+\int_ Z
E^{*}_t(\theta)H(\theta)\nu(d\theta)\\
&+\int_ Z
E^{*}_t(\theta)KE_t(\theta)\nu(d\theta)+\int_ Z
E^{*}_t(\theta)H(\theta)E_t(\theta)\nu(d\theta)\\
&+\displaystyle\sum_{i=1}^{d_2}(C^{2i}_t-D^{2i}_tU)^*K(C^{2i}_t-D^{2i}_tU).\\
    \end{split}
  \end{eqnarray}
 By notation \eqref{eq:4.18}, define the mapping $\hat{U}:[0, T]\times({\cal S}^n)^+\rightarrow R^{m\times
  n}$ by
$$\hat{U}(t, K)=\hat{N}^{-1}(t, K)\hat{B}^*(t, K).$$

\begin{lem}\label{lem:f346} Let $(K, L^1, H)\in ({\cal S}^n)_+\times ({\cal S}^n)^{d_1}\times {\cal L}^{\nu, 2}( Z;{\cal S}^n)$. Then, for $\forall U\in R^{m\times n}$,
we have
\begin{equation}
F(t, K, L^{1}, H, U)+{U}^*{N}_t{U}\geqslant F(t, K, L^{1}, H, \hat{U}(t,
K))+\hat{U}^*(t, K){N}_t\hat{U}(t, K),~~~0\leq t\leq T.
\end{equation}
\end{lem}

\begin{proof}
  By the definition of $F(t, K, L^{1}, H, U), \hat{B}(t, K), \hat{N}(t, K) $ and $\hat{U}(t, K)$, it follows that
\begin{eqnarray}
\begin{split}
F(t, K, L^{1}, H, U)&+{U}^*{N}_t{U}=-U^*\hat{B}^*(t, K) -\hat{B}(t,
K)U+U^*\hat{N}(t, K)U
\\&+A_t^*K+KA_t+\displaystyle\sum_{i=1}^{d_1}C^{1i*}_tL^{1i}
+\displaystyle\sum_{i=1}^{d_1}L^{1i}C^{1i}_t\\
&+\displaystyle\sum_{i=1}^{d_1}C^{1i*}_tKC^{1i}_t+\int_ ZH(\theta)E_t(\theta)\nu(d\theta)
+\int_ Z
E^{*}_t(\theta)H(\theta)\nu(d\theta)\\
&+\int_ Z
E^{*}_t(\theta)KE_t(\theta)\nu(d\theta)+\int_ Z
E^{*}_t(\theta)H(\theta)E_t(\theta)\nu(d\theta)\\
&+\displaystyle\sum_{i=1}^{d_2}C^{2i*}_tKC^{2i}_t,
\end{split}
 \end{eqnarray}
 and
  \begin{eqnarray}
    \begin{split}
   F(t, K, L^{1}, H, \hat{U}(t, K))&+\hat{U}^*(t, K){N}_t\hat{U}(t,K)=
-\hat{U}^*(t, K)\hat{N}(t, K)\hat{U}(t, K)
\\&+A_t^*K+KA_t+\displaystyle\sum_{i=1}^{d_1}C^{1i*}_tL^{1i}
+\displaystyle\sum_{i=1}^{d_1}L^{1i}C^{1i}_t\\
&+\displaystyle\sum_{i=1}^{d_1}C^{1i*}_tKC^{1i}_t+\int_ ZH(\theta)E_t(\theta)\nu(d\theta)
+\int_ Z
E^{*}_t(\theta)H(\theta)\nu(d\theta)\\
&+\int_ Z
E^{*}_t(\theta)KE_t(\theta)\nu(d\theta)+\int_ Z
E^{*}_t(\theta)H(\theta)E_t(\theta)\nu(d\theta)\\
&+\displaystyle\sum_{i=1}^{d_2}C^{2i*}_tKC^{2i}_t
.
    \end{split}
  \end{eqnarray}
Therefore,
  \begin{eqnarray}\label{eq:f366}
\begin{split}
&~~~F(t, K, L^{1}, H, U)+{U}^*{N}_t{U}- F(K, L^{1}, H, \hat{U}(t,
K))-\hat{U}^*(t, K){N}_t\hat{U}(t, K)
\\&=-U^*\hat{B}^*(t, K)
-\hat{B}(t, K)U+U^*\hat{N}(t, K)U+\hat{U}^*(t, K)\hat{N}(t, K)\hat{U}(t, K)\\
&=-U^*\hat{N}(t, K)\hat{U}(t, K)
-\hat{U}^*(t, K)\hat{N}_tU+U^*\hat{N}(t, K)U+\hat{U}^*(t, K)\hat{N}_t\hat{U}(t, K)\\
&=(\hat{U}(t, K)-U)^*\hat{N}(t, K)(\hat{U}(t, K)-U)\geqslant 0.
\end{split}
  \end{eqnarray}
  The proof is complete.
\end{proof}
In the following we will use the  Bellman¡¯s principle of
quasi-linearization and a monotone convergence result of symmetric matrices (see \cite{WONH68}) to show Theorem\ref{thm:f344}.

\emph{Existence}£º By the definition of $F(t, K,  L^1,  H, U)$, BSRDEJ \eqref{eq:4.17} can be rewritten as the following
 quasi-linearization BSDE
\begin{eqnarray}\label{eq:f367}
   \begin{split}
-dK_t=&\bigg[F(t, K_t, L^1_t, H_t, \hat{U}(t, K_t))
+\hat{U}^*(t, K_t){N}_t\hat{U}(t, K_t)+Q_t\bigg]dt\\
& -\displaystyle\sum_{i=1}^{d_1}L^{1i}_tdW^{1i}_t-\int_ ZH_t(\theta)\tilde{\mu}(d\theta,
dt), ~~~K_T=M.
\end{split}
  \end{eqnarray}
Making use of  Eq.\eqref{eq:f367}, we will iteratively construct
 a sequence $\{( K_{j}, L^1_{j}, H_{j})\}_{j=1}^{\infty}$of approximating solutions of BSRDEJ \eqref{eq:4.17}.
 In fact, by Lemma \ref{lem:f345}, we set  $(K_0, L_0^1, H_0)=(0,0,0)$ and solve iteratively the following
 linear BSDE:

 \begin{equation}\label{eq:f368}
\left\{\begin {array}{ll}
-dK_{j+1, t}=&\bigg[F(t, K_{j+1, t}, L^1_{j+1, t}, H_{j+1, t},
\hat{U}(t, K_{j, t}))
+\hat{U}^*(t, K_{j, t})N_t\hat{U}(t, K_{j, t})+Q_t\bigg]dt\\
&-\displaystyle\sum_{i=1}^{d_1}L^{1i}_{j+1, t}dW^{1i}_t
-\int_ ZH_{j+1, t}(\theta)\tilde{\mu}(d\theta, dt), \\
K_{j+1, T}=&M.
\end{array}
\right.
\end{equation}

From Lemma \ref{lem:f345}, it follows that $K_{j}$
is a.e.a.s. bounded and nonnegative.  We also claim that  $\{K_{j+1}\}$ is a.e.a.s. a non-increasing
sequence. Indeed,

\begin{eqnarray}
\begin{split}
&~~~~-d(K_{j, t}-K_{j+1, t})=(-dK_{j, t})-(-dK_{j+1, t})\\
&=\bigg[F(t, K_{j, t}, L_{j, t}^1, H_{j, t}, \hat{U}(t, K_{j-1, t}))
+\hat{U}^*(t, K_{j-1, t})N_t\hat{U}(t, K_{j-1, t}))\\
&~~~~-F(t, K_{j+1, t}, L_{j+1, t}^1, H_{j+1, t}, \hat{U}(t, K_{j,
t}))
-\hat{U}^*(t, K_{j, t})N_t\hat{U}(t, K_{j, t})\bigg]dt\\
&~~~-\displaystyle\sum_{i=1}^{d_1}(L_{j, t}^{1i}-L_{j+1,
t}^{1i})dW_t^{1i}-\int_Z(H_{j, t}(\theta)
-H_{j+1, t}(\theta))\tilde\mu(d\theta, dt)\\
&=\bigg[F(t, K_{j, t}, L_{j, t}^1, H_{j, t}, \hat{U}(t, K_{j, t}))
-F(K_{j+1, t}, L_{j+1, t}^1, H_{j+1, t}, \hat{U}(t, K_{j, t}))\\
&~~~+\bigg(F(t, K_{j, t}, L_{j, t}^1, H_{j, t}, \hat{U}(t, K_{j-1,
t}))
+\hat{U}^*(t, K_{j-1, t})N_t\hat{U}(t, K_{j-1, t}))\\
&~~~-F(t, K_{j, t}, L_{j, t}, H_{j, t}, \hat{U}(t, K_{j, t}))
-\hat{U}^*(t, K_{j, t})N_t\hat{U}(t, K_{j, t})\bigg)\bigg]dt\\
&~~~-\displaystyle\sum_{i=1}^{d_1}(L_{j, t}^{1i}-L_{j+1,t}^{1i})dW_t^{1i}-\int_Z(H_{j, t}(\theta)
-H_{j+1, t}(\theta))\tilde\mu(d\theta, dt)\\
&=\bigg[F(t, K_{j, t}-K_{j+1, t}, L_{j, t}-L_{j+1, t}, H_{j, t}
-H_{j+1, t}, \hat{U}(t, K_{j, t}))\\
&~~~+(\hat{U}(t, K_{j, t})-\hat{U}(t, K_{j-1, t}))^*\hat{N}(t, K_{j,
t})
(\hat{U}(t, K_{j, t})-\hat{U}(t, K_{j-1, t}))\bigg]dt\\
&~~~-\displaystyle\sum_{i=1}^{d_1}(L_{j, t}^{1i}-L_{j+1,
t}^{1i})dW_t^{1i}-\int_Z(H_{j, t}^{1}(\theta) -H_{j+1,
t}^{1}(\theta))\tilde\mu(d\theta, dt),
\end{split}
  \end{eqnarray}
where we  have  used the equality   \eqref{eq:f366} in Lemma \ref{lem:f346}.

Since $(\hat{U}(t, K_{j, t})-\hat{U}(t, K_{j-1, t}))^*\hat{N}(t, K_{j,
t}) (\hat{U}(t, K_{j, t})-\hat{U}(t, K_{j-1,
t}))$ is nonnegative, according to Lemma \ref{lem:f345},
 we conclude that $K_{j,
t}-K_{j+1, t}$ is also nonnegative. This implies $\{K_{j}\}_{j=1}^\infty$ is a non-increasing sequence

$$
CI\geqslant K_{1, t}\geqslant K_{2, t}\geqslant \cdots\geqslant K_{j, t}\geqslant \cdots\geqslant
0,~~~~~t\in [0, T].
$$
It follows that $\{K_{j}\}$ converges almost surely to a nonnegative bounded, $S^n$-
valued process $K.$
According to Lebesgue`s convergence theorem, we have
  \begin{eqnarray}\label{eq:533}
    \begin{split}
\displaystyle\lim_{j\rightarrow\infty}E\int_0^T|K_{j,
t}-K_t|^qdt\rightarrow 0, ~~\forall q>0.
\end{split}
  \end{eqnarray}
Thus $\{K_{j}\}_{j=1}^\infty$ and $\{u(t, K_{j})\}_{j=1}^\infty$ is a Cauchy sequence in the above sense.
Again using Lebesgue`s convergence theorem, for $t\in [0, T],$
we also have

 \begin{eqnarray}\label{eq:534}
    \begin{split}
\displaystyle\lim_{j\rightarrow\infty}E|K_{j, t}-K_t|^q\rightarrow
0, ~~~~\forall q>0.
\end{split}
  \end{eqnarray}

 Applying It\^{o}'s formula  to $|K_{k, t}-K_{j,
t}|^2$, we get
\begin{eqnarray}\label{eq:f371}
\begin{split}
&E|K_{k,0}-K_{j,0}|^2+\sum_{i=1}^{d_1}E\int_0^T|L_{k, t}^{1i}-L_{j, t}^{1i}|^2dt
+E\iint_{ Z\times(0, T] }|H_{k, t}(\theta)-H_{j, t}(\theta)|^2\nu(d\theta)dt\\
=&2E\int_0^Ttr\bigg[(K_{k, t}-K_{j, t})\bigg(C_t^{1i*}(L_{k, t}^{1,
i}-L_{j, t}^{1i})
+(L_{k, t}^{1i}-L_{j, t}^{1i})C_t^{1i}\\
&+\int_ Z E_t^{*}(\theta)(H_{k, t}(\theta)-H_{j,
t}(\theta))\nu(d\theta) +\int_ Z(H_{k, t}(\theta)-H_{j,
t}(\theta))E_t(\theta)\nu(d\theta)
\\&+\int_ Z E_t^{*}(\theta)
(H_{k, t}(\theta)-H_{j, t}(\theta))E_t(\theta)\nu(d\theta)dt+R(j, k)\\
\leqslant &\frac{1}{2} \displaystyle\sum_{i=1}^{d_1} E\int_0^T|L_{k,
t}^{1i}-L_{j, t}^{1i}|^2dt +\frac{1}{2}E\iint_{ Z\times (0,
T]}|H_{k, t}(\theta)
-H_{j, t}(\theta)|^2\nu(d\theta)dt\\
&+CE\int_0^T|K_{k, t}-K_{j, t}|^2dt,
\end{split}
  \end{eqnarray}
where \begin{eqnarray}
  \begin{split}
    R(j, k)= &2\displaystyle\int_0^Ttr\bigg[ (K_{k, t}-K_{j, t})\bigg((K_{k, t}-K_{j, t})A_t^*+ A_t^*(K_{k, t}-K_{j, t})\\&+\displaystyle\sum_{i=1}^{d_2}C^{2i*}_t(K_{k,t}-K_{j,t})C^{2i}_t
    +\hat{U}^*(t, K_{k-1,t})\hat{N}(t, K_{k,t})\hat{U}(t, K_{k-1,t})
    \\&-\hat{U}^*(t, K_{j-1,t})\hat{N}(t, K_{j,t})\hat{U}(t, K_{j-1,t})
    -U^*(t,K_{k-1,t})\hat{B}^*(t, K_{k,t})
    \\&-\hat{B}(t,K_{k,t})U(t,K_{k-1,t})+U^*(t,K_{j-1,t})\hat{B}^*(t, K_{j,t}) +\hat{B}(t,K_{j,t})U(t,K_{j-1,t})\bigg]dt
    \end{split}
\end{eqnarray}
and $C$  is some deterministic positive constant.

Thus from \eqref{eq:533} and \eqref{eq:534}, we know that
 $\{L_{j}^{1}\}_{j=1}^\infty$ and $ \{H_j\}_{j=1}^\infty$ is
 Cauchy sequences in  ${\cal
L}_{\cal F^*}^2((0, T;({S}^n)^{d_1})$ and  ${\cal L}_{\cal F^*}^{\nu,
2}([0, T]\times
 Z;{S}^n)$ respectively. We denote the limits  by $L^1$ and $H$ respectively.
In the end, passing the limit in Eq.\eqref{eq:f368}, we obtain that $(K, L^1, H)$
satisfies Eq.\eqref{eq:f367}. Thus $(K, L^1,
H)$ is the solution of BSRDEJ \eqref{eq:4.17}.

\emph{Uniqueness}

 Suppose that  BSRDEJ \eqref{eq:4.17} has  two solutions $(K, L^1, H)$ and $(\bar{K},  \bar{L}^1,  \bar{H})$  .
 Then it follows from \eqref{eq:f367} that
\begin{eqnarray*}
\begin{split}
-dK_t=&\bigg[F(t, K_t, L^1_t, H_t, \hat{U}(t, K_t))
+\hat{U}^*(t, K_t){N}_t\hat{U}(t, K_t)+Q_t\bigg]dt\\
&~~~~~~~ -\displaystyle\sum_{i=1}^{d_1}L^{1i}_tdW^{1i}_t-\int_ ZH_t(\theta)\tilde{\mu}(d\theta,
dt),
~~K_T=M
\end{split}
  \end{eqnarray*}
and
\begin{eqnarray*}
\begin{split}
-d\bar{K}_t=&\bigg[F(t, \bar{K}_t, \bar{L}^1_t, \bar{H}_t,
\hat{U}(t, \bar{K}_t))
+\hat{U}^*(t, \bar{K}_t){N}_t\hat{U}(t, \bar{K}_t)+Q_t\bigg]dt\\
&~~~~~~~ -\displaystyle\sum_{i=1}^{d_1}\bar{L}^{1i}_tdW^{1i}_t-\int_ Z\bar{H}_t(\theta)\tilde{\mu}(d\theta,
dt),
~~\bar{K}_T=M,
\end{split}
  \end{eqnarray*}
Thus
\begin{eqnarray}
\begin{split}
-d(\bar{K}_t-K_t)=&\bigg[F(t, \bar{K}_t-K_t, \bar{L}_t^1-L_t^1,
\bar{H}_t-H_t, \hat{U}(t, \bar{K}_t))\\
&+F(t, K_t, L^1_t, H_t, \hat{U}(t, \bar{K}_t))
+\hat{U}^*(t, \bar{K}_t){N}_t\hat{U}(t, \bar{K}_t)\\
&-F(t, K_t, L^1_t, H_t, \hat{U}(t, K_t))
-\hat{U}^*(t, K_t)N_t\hat{U}(t, K_t))\bigg]dt\\
&-\displaystyle\sum_{i=1}^{d_1}(\bar{L}_t^{1i}-L_t^{1i})dW_t^{1i}
-\int_ Z(\bar{H}_t(\theta)-H_t(\theta))\tilde{\mu}(d\theta, dt)\\
=&\bigg[F(t, \bar{K}_t-K_t, \bar{L}_t^1-L_t^1,  \bar{H}_t-H_t,
\hat{U}(t, \bar{K}_t))
\\&+(\hat{U}(t, K_t)-\hat{U}(t, \bar{K}_t))^*\hat{N}(t, K_t)(\hat{U}(t, K_t)
-\hat{U}(t, \bar{K}_t))\bigg]dt\\
&-\displaystyle\sum_{i=1}^{d_1}(\bar{L}_t^{1i}-L_t^{1})dW_t^{1i}
-\int_ Z(\bar{H}_t(\theta)-H_t(\theta))\tilde{\mu}(d\theta, dt), \\
~~~~~~~~\bar{K}_T-K_T=0.&
\end{split}
  \end{eqnarray}
Since $(\hat{U}(t, K_t)-\hat{U}(t, \bar{K}_t))^*\hat{N}(t,
K)(\hat{U}(t, K_t) -\hat{U}(t, \bar{K}_t))$ is nonnegative, it follows from Lemma
\ref{lem:f345} that  $\bar{K}-K$ is also a.e. a.s. nonnegative.
Similarly we can obtain that  $\bar{K}-K$ is a.e.a.s. nonnegative.  This implies $K=\bar{K}.$
In the end, from the uniqueness result of  Lemma \ref{lem:f345}, we conclude that $\bar{L}^{1}=L^{1}, \bar{H}=H.$
The uniqueness is proved.

\end{proof}
 { \bf Acknowledgements}
The author is very grateful to Professor Tang shanjian  for his valuable
suggestions and various instruction.


\end{document}